\documentclass[leqno,draft]{amsart}

\usepackage{amsmath,amsfonts,amssymb,amsthm}

\newcommand{\N}{\mathbb{N}}
\newcommand{\Z}{\mathbb{Z}}

\newcommand{\R}{\mathbb{R}}

\newcommand{\PP}{\mathcal{P}}
\newcommand{\inte}{\mathrm{int}}

\newcommand{\id}{\mathrm{id}}
\newcommand{\conv}{\mathrm{conv}}
\newcommand{\ar}{\rightarrow}

\newtheorem{thm}{Theorem}[section]
\newtheorem{prop}[thm]{Proposition}
\newtheorem{lem}[thm]{Lemma}
\newtheorem{cor}[thm]{Corollary}

\theoremstyle{definition} 
\newtheorem{assu}[thm]{Assumption}
\newtheorem{dfn}[thm]{Definition}
\newtheorem{rem}[thm]{Remark}
\newtheorem{obs}[thm]{Observation}
\newtheorem{exa}[thm]{Example}

\title[Cannon-Thurston maps for Coxeter groups with signature $(n-1,1)$]
{Cannon-Thurston maps for Coxeter groups \\ with signature $(n-1,1)$}
\author[Ryosuke Mineyama]{Ryosuke Mineyama}

\thanks{
\hspace{-0.5cm}{\bf 2010 Mathematics Subject Classification:} 
Primary 20F55, 51F15; Secondary 05E15.  }
\address{Ryosuke Mineyama, 
Department of Mathematics, 
Graduate School of Science,
Osaka University,
Toyonaka, Osaka 560-0043, Japan}
\email{r-mineyama@cr.math.sci.osaka-u.ac.jp}
\begin{document}

\begin{abstract}
For a Coxeter group $W$ we have an associating bi-linear form $B$ on suitable real vector space.
We assume that $B$ has the signature $(n-1,1)$ and all the bi-linear form 
associating rank $n' (\ge 3)$
Coxeter subgroups generated by subsets of $S$ has the signature $(n',0)$ or $(n'-1,1)$.
Under these assumptions, we see that there exists the Cannon-Thurston map for $W$,
that is, the $W$-equivariant continuous surjection from the Gromov boundary of $W$ to the limit set 
of $W$.
To see this we construct an isometric action of $W$ on an ellipsoid with 
the Hilbert metric.
As a consequence, we see that the limit set of $W$ coincides with the set of accumulation points
of roots of $W$.
\end{abstract}

\maketitle


\section{Introduction}
A new dynamical approach to analyze the asymptotic behavior of the 
root system associating a Coxeter group
has been introduced by Hohlweg, Labb\'e and Ripoll in \cite{hlr}.
This approach implicate a study of infinite Coxeter groups
from a dynamical viewpoint. 
In this paper we analyze the asymptotic behavior of 
the orbit of a point under the action of a infinite Coxeter group not only for the roots.
As is known in the theory of the Kleinian groups,
to study accumulation points is nothing 
but to study the interaction between ergodic theory and discrete groups.
In order to establish that theory, the hyperbolicity of its phase space plays a crucial role.
For the case where the associated matrices have signature $(n-1,1)$, 
Coxeter groups also act on hyperbolic space in the sense of Gromov.
In \cite{hpr}, the authors argue the connection between the theory of the Klieinian
groups and the Coxeter groups via the hyperbolic geometry.
They investigated isometrical actions on hyperbolic spaces and 
showed the limit sets of Coxeter systems of type $(n-1,1)$ coincide with
the set of accumulation points of its roots. 
In this paper, we also focus on our attention to infinite Coxeter groups 
whose bi-linear form $(n-1,1)$.

\subsection{Known results and Motivation}
In general,
a continuous equivariant between boundaries of a discrete group and 
their limit set is called a Cannon-Thurston map.
The existence of such map is one of the most interesting question 
in the group theory from a geometrical viewpoint.
For the Kleinian groups several authors contributed to this topic.
In particular recently Mj showed that for Kleinan surface groups 
(in fact for all finitely generated Kleinian groups) 
there exist the Cannon-Thurston maps and local connectivity of the connected limit sets \cite{Mj}.

From more general point of view Mitra considered
the Cannon-Thurston map for Gromov hyperbolic groups.
Let $H$ be a hyperbolic subgroup of a hyperbolic group $G$ in the sense of Gromov.
He asked whether the inclusion map always extends continuously to the equivariant map between the 
Gromov compactifications $\widehat{H}$ and $\widehat{G}$.
For this question he positively answered in the case when $H$ is
an infinite normal subgroup of a hyperbolic group $G$ \cite{Mit1}.
He also proved that
the existence of the Cannon-Thurston map when $G$ is
a hyperbolic group acting cocompactly on a simplicial tree $T$ 
such that all vertex and edge stabilizers are hyperbolic,
and $H$ is the stabilizer of a vertex or edge of $T$ provided 
every inclusion of an edge stabilizer in a vertex stabilizer is a quasi isometric embedding
\cite{Mit2}.
On the other hand, Baker and Riley constructed a negative example for Mitra's question.
In fact they proved that there exits a free subgroup of rank 3 in 
a hyperbolic group such that the Cannon-Thurston map is not well-defined \cite{BR}.
Adding to this Matsuda and Oguni showed that 
a similar phenomenon occurs for every non-elementary relatively hyperbolic group
\cite{MO}.

Inspired by the above results we shall consider the problem which asks whether 
the Cannon-Thurston map for the Coxeter groups exists.

\subsection{Results in this paper}
In this paper we prove the following.

\begin{thm}\label{main}
	Let $W$ be a Coxeter group of rank $n$ whose associating bi-linear form $B$ has
	signature $(n-1,1)$ and $S$ be its generating set.
	Let $\partial_G W$ be the Gromov boundary of $W$ and let 
	$\Lambda(W)$ be the limit set of $W$.
	If all the bi-linear form  associating rank $n' (\ge 3)$
	bi-linear forms of Coxeter subgroups generated by subsets of $S$ 
	are positive definite or have signature $(n'-1,1)$,
	then we have $W$-equivariant, continuous surjection 
	$F : \partial_G W \ar \Lambda(W)$.
\end{thm}

More of this we can easily see that the limit set of an arbitrary Coxeter subsystem $(W',S')$ 
of $(W,S)$ such that $S'\subset S$ is a section of the limit set of $W$ by a hyperplane.
Hence the limit set of $W'$ is identically included in the limit set of $W$.

The reason can be found in Section 4 
why the assumption of our theorem is not only for the signature of $B$
but also for its principal submatrices.
For the most general case we need more complicated arguments.
We will discuss excepted cases in the forthcoming paper.

Under our assumption for the bi-linear forms,
we will see that many of Coxeter groups of type $(n-1,1)$ are Gromov hyperbolic.
For the Gromov hyperbolic groups acting co-compactly,
the existence of the Cannon-Thurston maps between their Gromov boundaries
immediately follows from the well known fact that 
such groups are quasi isometric to their phase spaces.
However since our situation is slightly different,
it is non-trivial the existence of Cannon-Thurston maps.

In \cite{dhr} and \cite{hmn} it has proved a conjecture
proposed in \cite[Conjecture 3.9]{hlr} for the case where the associating 
bilinear form has the signature $(n-1,1)$.  
That states the distribution
of accumulation points of roots of infinite Coxeter groups can be described as some
appropriate set of points.
It is natural to compare the limit set and the set of accumulation points of roots.
As a consequence, we also prove the same the result in \cite[Theorem 1.1]{hpr}.
We remark that this has been done independently.

\begin{thm}\label{main2}
	Let $W$ be a Coxeter group of rank $n$ whose associating bi-linear form $B$ has
	signature $(n-1,1)$.
	Then the limit set $\Lambda(W)$ of $W$ coincides with 
	the set of accumulation points of roots $E(W)$ of $W$.
\end{thm}

A brief outline of the paper follows.
In Section 2, we recall some terminologies in the theory of Coxeter groups and 
define our action.
In Section 3, we define Hilbert metrics on ellipsoids on which the
corresponding Coxeter groups act properly and 
correct some basic properties of the Hilbert metric spaces
without the proofs.
The properness of our action is proved in Section 4.
We also give the proof of Theorem \ref{main2} there.
In Section 5,
we define Gromov and CAT(0) boundaries of metric spaces.
In Section 6,
we show our main result (Theorem \ref{main}).

\subsection*{Acknowledgements} 
The author would like to be grateful to Prof. Hideki Miyachi 
for his helpful comments and suggestion to work on Theorem \ref{main}. 
The author would also thank to Prof. Yohei Komori for insightful comments and 
giving him an example.


\section{The Coxeter systems and geometric representation}

\subsection{The Coxeter systems}
A {\em Coxeter group} $W$ of rank $n$ is generated by the set $S=\{s_1, \ldots, s_n\}$ 
with the relations $(s_is_j)^{m_{ij}}=1$, 
where $m_{ij} \in \Z_{>1} \cup \{\infty\}$ for $1 \leq i < j \leq n$ 
and $m_{ii}=1$ for $1 \leq i \leq n$. 
More precisely, we say that the pair $(W,S)$ is a {\em Coxeter system}. 

For a Coxeter system $(W,S)$ of rank $n$, 
let $V$ be a real vector space with its orthonormal basis $\Delta=\{\alpha_s \vert s \in S\}$
with respect to the Euclidean inner product. 
Note that by identifying $V$ with $\R^n$, we treat $V$ as a Euclidean space. 
We define a symmetric bilinear form on $V$ by setting 
\begin{align*}
B(\alpha_i, \alpha_j) \; 
	\begin{cases}
	\ =-\cos\left(\frac{\pi}{m_{ij}}\right) \;\;\;&\text{if } m_{ij}< \infty, \\
	\ \leq -1 &\text{if }m_{ij}=\infty 
	\end{cases}
\end{align*}
for $1 \leq i \leq j \leq n$, where $\alpha_{s_i}=\alpha_i$, 
and call the associated matrix $B$ the {\em Gram matrix}. 
Classically, $B(\alpha_i, \alpha_j)=-1$ if $m_{ij}=\infty$, 
but throughout this thesis, we allow its value to be any real number less than or equal to $-1$. 
This definition derives from \cite{hlr}. 
Given $\alpha \in V$ such that $B(\alpha, \alpha) \not=0$, $s_\alpha$ denotes 
the map $s_\alpha : V \to V$ by
\[
	s_\alpha(v)=v- 2 \frac{B(\alpha,v)}{B(\alpha, \alpha)}\alpha \;\;\;\text{for any } v \in V,
\] 
which is said to be a {\em $B$-reflection}. 
Then $\Delta$ is called a {\em simple system} and its elements are {\em simple roots} of $W$.
The Coxeter group $W$ acts on $V$ associated with its generating set $S$ 
as compositions of $B$-reflections $\{s_{\alpha}\ \vert\ \alpha \in \Delta\}$ 
generated by simple roots.
The {\em root system} $\Phi$ of $W$ is defined to be the orbit of $\Delta$ 
under the action of $W$ and its elements are called its {\em roots}. 
Let 
\begin{align*}
	V^+ := 
		\left\{\ v \in V\ \left\vert\ v=\sum_{i=1}^n v_i\alpha_i, v_i > 0\right.\right\},\ 
	V^- := 
		\left\{\ v \in V\ \left\vert\ v=\sum_{i=1}^n v_i\alpha_i, v_i < 0\right.\right\}. 
\end{align*}

\begin{assu}\label{assume}
	In this paper, we always assume the following.
	\begin{itemize}
	\item The bilinear form $B$ has the signature $(n-1,1)$.
	We call such a group a Coxeter group of type $(n-1,1)$.
	\item The Gram matrix$B$ is not block-diagonal 
	up to permutation of the basis.
	In that case, the matrix $B$ is said to be {\it irreducible}.
	\end{itemize}
\end{assu}

Recall that a matrix $A$ is {\it non-negative} 
if each entry of $A$ is non-negative.

\begin{lem}\label{heimen}\label{korekore}
	Let $o$ be an eigenvector for the negative eigenvalue of $B$.
	Then all coordinates of $o$ have the same sign.
\end{lem}

\begin{proof}
	This follows from Perron-Frobenius theorem for irreducible 
	non-negative matrices.
	Let $I$ be the identity matrix of rank $n$.
	Then $-B+I$ is irreducible and non-negative.
	Note that since $-B+I$ and $B$ are symmetric, all eigenvalues are real.
	By Perron-Frobenius theorem, we have a positive eigenvalue $\lambda'$ of $-B+I$
	such that $\lambda'$ is the maximum of eigenvalues of $-B+I$ and 
	each entry of corresponding eigenvector $u$ is positive.
	On the other hand, for each eigenvalue $a$ of $B$ there exists an eigenvalue $b$ of $-B+I$
	such that $a = 1-b$.
	Let $\lambda$ be the negative eigenvalue of $B$.
	Then an easy calculation gives $\lambda = 1-\lambda'$.
	Therefore $\R u = \R o$.
\end{proof}

We fix $o \in V$ to be the eigenvector corresponding to the negative eigenvalue of $B$ 
whose euclidean norm equals to $1$ and all coordinates are positive.
Hence if we write $o$ in a linear combination $o = \sum_{i=1}^n o_i \alpha_i$ of $\Delta$
then $o_i > 0$.
Given $v \in V$, we define $|v|_1$ by $\sum_{i=1}^no_iv_i$ if $v = \sum_{i=1}^n v_i \alpha_i$. 
Note that a function $| \cdot |_1 : V \to \R$ is actually a norm in the set of vectors 
having nonnegative coefficients. 
It is obvious that $|v|_1>0$ for $v \in V^+$ and $|v|_1<0$ for $v \in V^-$. 
Let $V_i=\{ v \in V\ \vert\ |v|_1=i\}$, where $i=0,1$. 
For $v \in V \setminus V_0$, 
we write $\widehat{v}$ for the ``normalized'' vector $\frac{v}{|v|_1} \in V_1$. 
We also call $o$ the normalized eigenvector (corresponding to the negative eigenvalue of $B$).
Also for a set $A \subset V \setminus V_0$, 
we write $\widehat{A}$ for the set of all $\widehat{a}$ with $a \in A$.
We notice that $B(x,\alpha) = |\alpha|_1B(x,\widehat{\alpha})$ hence 
the sign of $B(x,\alpha)$ equals to the sign of $B(x,\alpha)$ for any 
$x \in V$ and $\alpha \in \Delta$.

\begin{rem}\label{daijoubu}
All roots are contained in $V^+ \cup V^-$ and hence 
$\Phi \cap V_0 = \emptyset$.
\end{rem}

Then by Remark\ref{daijoubu}, the set $\widehat{\Phi}$ is well-defined.
Let $E$ be the set of accumulation points of $\widehat{\Phi}$ 
with respect to the Euclidean topology.

It turns out that we only need to work on the case where $B$ is irreducible.
If the matrix $B$ is reducible, then we can divide $\Delta$
into $l$ subsets $\Delta = \sqcup_{i=1}^l\Delta_i $
so that each corresponding matrix $B_{i} = \{B(\alpha,\beta)\}_{\alpha,\beta \in \Delta_i}$ 
is irreducible and $B$ is block diagonal $B = (B_1,\ldots,B_l)$.
Then for any distinct $i,j$, if $\alpha \in \Delta_i$ and $\beta \in \Delta_j$,
$s_\alpha$ and $s_\beta$ commute.
In this case we see that $W$ is direct product
\[
	W = W_1 \times W_2 \times \cdots \times W_l,
\]
where $W_i$ is the Coxeter group corresponding to $\Delta_i$.
From this, the action of $W$ can be regarded as a direct product of the actions of each $W_i$.
Then for the set $E$ of accumulation points of roots of $W$ we see that 
$
	E = \sqcup_{i=1}^l E_i,
$
where $E_i$ is the set of accumulation points of roots $W_i \cdot \Delta_i$ 
(see Proposition 2.14 in \cite{hlr}).
Moreover if $B$ has the signature $(n-1,1)$, there exists a unique $B_k$ 
which has the signature $(n_k-1,1)$ and others are positive definite.
Since if the Gram matrixis positive definite then 
the corresponding Coxeter group $W'$ is finite, and hence 
the limit set $\Lambda(W') = \emptyset$ (for the definition of the limit set, see Section 3.3).
This ensures that $\Lambda(W)$ is distributed on $\conv(\widehat{\Delta_k})$,
where $\conv(\widehat{\Delta_k})$ is the convex hull of $\widehat{\Delta_k}$.
Thus $\Lambda(W) = \Lambda(W_k)$.
Accordingly, 
if there exists the Cannon-Thurston map for $W_k$ then we also have the Cannon-Thurston map
for the whole group $W$.
This follows from the fact that the direct product $G_1 \times G_2$ of 
a finite generated infinite group $G_1$ and a finite group $G_2$
has the same Gromov boundary as that of $G_1$.

We denote $q(v)=B(v,v)$ for $v \in V$. 
Let $Q=\{v \in V\ \vert\  q(v)=0\}$, $Q_- = \{ v \in V\ \vert\ q(v)<0\}$ then we have
\[
	\widehat{Q} = V_1 \cap Q,\ \ \ \widehat{Q_-} = V_1 \cap Q_-.
\]
Since $B$ is of type $(n-1,1)$, $\widehat{Q}$ is an ellipsoid.
The cone $Q_-$ has two components the ``positive side'' $Q_-^+$, 
that is the component including $o$,
and the ``negative side'' $Q_-^- = - Q_-^+$.
Similarly we divide $Q$ into two components $Q^+$ and $Q^-$ so that
$Q^+ = \partial Q_-^+$ and $Q^- = \partial Q_-^-$.

\begin{rem}\label{zerodake} 
We have 
\[
	W(V_0) \cap Q = \{{\bf 0}\},
\] 
where {\bf 0} is the origin of $\R^n$.
To see this we only need to verify that $V_0 \cap Q = \{{\bf 0}\}$ since $Q$ is invariant under
$B$-reflections.
We notice that $V_0 = \{v \in V\ \vert\ B(v,o) = 0\}$.
For $i = 1,\ldots,n-1$, 
let $p_i$ be an eigenvector of $B$ corresponding to a positive eigenvalue $\lambda_i$.
For any $v \in V_0$, we can express $v$ in a linear combination 
$v = \sum_i^{n-1} v_i p_i$ since $B(v,o) = 0$.
Then we have $B(v,v) = \sum_i^{n-1} \lambda_i v_i^2 \|p_i\|^2\ge 0$
where $\|*\|$ denotes the euclidean norm.
Since $\lambda_i > 0$ for $i = 1,\ldots,n-1$, we have $B(v,v) = 0$ if and only if $v={\bf 0}$. 
\end{rem}


\subsection{The word metric}
This paper is devoted to the connection between the geometry of the Coxeter groups themselves
and their acting spaces.
To do this, it needs to regard the groups as metric spaces.

Let $G$ be a finitely generated group.
Fixing a finite generating set $S$ of $G$,
all elements in $G$ can be represented by a product of elements in $S \cup S^{-1}$
where $S^{-1} = \{s^{-1}\ \vert\ s \in S\}$.
We say such a representation to be a {\it word}.
Letting $\langle S \rangle$ be the set of words.
For a word $w \in \langle S \rangle$
we define the {\it word length} $\ell_{S}(w)$ as 
the number of generators $s \in S$ in $w$.
Now, we naturally have a map 
$\iota : \langle S \rangle \ar W$.
For a given $g \in G$, we define the {\it minimal word length} $|g|_S$ of $g$ by
$\min\{\ell_S(w)\ \vert\ w \in \iota^{-1}(g)\}$.
An expression of $g$ realizing $|g|_S$ is called {\it the reduced expression}
or {\it the geodesic word}.
Using the word length, we can define so-called {\it the word metric} with respect to $S$ on $G$, i.e.
for $g,h \in G$, their distance is $|g^{-1}h|_S$.


\section{The Hilbert metric}

\subsection{The cross ratio and the Hilbert metric}
For four vectors $a,b,c,d \in V$ with $c-d, b-a \notin Q$, 
we define the {\it cross\ ratio} $[a,b,c,d]$ with respect to $B$ by 
\[
	[a,b,c,d] := \frac{q(c-a)\cdot q(b-d)}{q(c-d)\cdot q(b-a)}.
\]
We observe how the cross ratio works in a cone.  

\begin{prop}\label{1}
	Let four points $a_1,a_2,a_3,a_4 \in V$ be collinear 
	(namely $a_2,a_3$ are on the segment connecting $a_1$ and $a_4$), 
	and $a_1-a_4 \notin Q$.
	Let $b_1,b_2,b_3,b_4 \in V$ satisfying 
	\begin{itemize}
		\item for each $i$, $b_i$ lies on a ray $R_i$
			  connecting $a_i$ and some point $p \in V$,
		\item four vectors $b_1,b_2,b_3,b_4$ are co-linear and $b_1-b_4 \notin Q$.
	\end{itemize}
	Then we have 
	\[
		[a_1,a_2,a_3,a_4] = [b_1,b_2,b_3,b_4].
	\]
\end{prop}

\begin{proof}
	By the assumption, 
	all eight points are located on the two dimensional subspace $P$ 
	which is spanned by $a_1-p$ and $a_4-p$ in $V$.
	
	Let $\ell_0$ be a line in $P$ through $a_1$ and $a_4$.
	Consider two lines $\ell_2$ and $\ell_3$ in $P$ parallel to $\ell_0$ 
	with $b_2 \in \ell_2$ and $b_3 \in \ell_3$.
	Let $B_i \in R_i \cap \ell_2$ and $B'_i \in R_i \cap \ell_3$ for $i = 1, 2, 3, 4$. 
	Then we have  $b_2 = B_2$ and $b_3 = B'_3$, and
	there is a positive constant $k$ such that $B'_i -p = k(B_i - p)$ for $i = 1, 2, 3, 4$.
	Since two triangles with vertices $\{b_4,b_2,B_4\}$ and $\{b_4,b_3,B'_4\}$ are similar,
	$$
		\frac{q(b_2-b_4)}{q(b_3-b_4)} = \frac{q(B_2-B_4)}{q(B'_3-B'_4)}.
	$$
	By the similar reason, we also have
	$$
		\frac{q(b_2-b_1)}{q(b_3-b_1)} = \frac{q(B_2-B_1)}{q(B'_3-B'_1)}.
	$$
	In addition since $\ell_0$ and $\ell_2$ are parallel, there exists 
	a constant $m$ so that
	$$
		B_i - B_j = m(a_i - a_j),
	$$ 
	for all $i,j \in \{1,2,3,4\}$.
	Therefore, we obtain
	\begin{align*}
		[b_1,b_2,b_3,b_4] &= \frac{q(b_3-b_1)q(b_2-b_4)}{q(b_3-b_4)q(b_2-b_1)}
						   = \frac{q(B'_3-B'_1)q(B_2-B_4)}{q(B'_3-B'_4)q(B_2-B_1)}\\
						  &= \frac{q(k(B_3-B_1))q(B_2-B_4)}{q(k(B_3-B_4))q(B_2-B_1)}
						   = \frac{q(B_3-B_1)q(B_2-B_4)}{q(B_3-B_4)q(B_2-B_1)}\\
						  &= \frac{q(a_3-a_1)q(a_2-a_4)}{q(a_3-a_4)q(a_2-a_1)}
						   = [a_1,a_2,a_3,a_4],
	\end{align*}
	which implies what we wanted.
\end{proof}

Using the cross ratio we define a distance $d_D$ on $D$ as follows.
For any $x,y \in D$, take $a,b \in \partial D$ 
so that the points $a,x,y,b$ lie on the segment connecting $a,b$ in this order.
Then $y-b,x-a \notin Q$.
We define
$$
	d_D(x,y) := \frac{1}{2} \log [a,x,y,b],
$$
and call this the {\it Hilbert metric for $B$}.
The definition of the Hilbert metric for $B$ depends heavily on $B$.
However following observation tells us that our definition coincides with 
the ordinary Hilbert metric $d_H$ on $D$.
Recall that the ordinary Hilbert metric $d_H$ on $D$ is defined 
for taking $a,x,y,b$ as above,
$$
	d_H(x,y) = \log \left( \frac{\|y-a\|\ \|x-b\|}{\|y-b\|\ \|x-a\|} \right)
$$
where $\|*\|$ denotes the Euclidean norm.

\begin{obs}
	Take arbitrary $x,y \in Q_-$ and pick two points $a,b \in \partial D$ up
	so that $d_D(x,y) = \frac{1}{2} \log [a,x,y,b]$.
	Then we have $\|y-b\| \le \|x-b\|$,\ $\|x-a\| \le \|y-a\|$. 
	From the collinearity, each pair $\{y-b,x-b\}$ and $\{x-a,y-a\}$ have same direction
	respectively. 
	Hence there exist constants $k,l \ge 1$ such that $x-b = k(y-b)$ and $y-a = l(x-a)$.
	Thus we have
	\begin{align*}
		[a,x,y,b] &= \frac{q(y-a)\ q(x-b)}{q(y-b)\ q(x-a)} 
				   = \frac{l^2q(x-a)\ k^2q(y-b)}{q(y-b)\ q(x-a)} \\
				  &= l^2\cdot k^2
				   = \left(\frac{l\|x-a\|\ k\|y-b\|}{\|y-b\|\ \|x-a\|}\right)^2
				   = \left(\frac{\|y-a\|\ \|x-b\|}{\|y-b\|\ \|x-a\|}\right)^2.
	\end{align*}
	This shows that $d_D(x,y) = d_H(x,y)$ for all $x,y \in D$.
\end{obs}

By this observation, we can call the Hilbert metric for $B$ merely the Hilbert metric.
An advantage of our definition of the Hilbert metric for $B$ will appear in the proof of
Proposition \ref{isom}.
Note that since $d_D = d_H$, the Hilbert metric is actually a metric.


\subsection{Some properties of the Hilbert metric} 
In this section we correct known geometric properties of a space with the Hilbert metric.

Let $(X,d)$ be a metric space.
We define {\it the length $len(\gamma)$} of an arc $\gamma:[0,t] \ar (X,d)$ by
$$
	len(\gamma) = \sup_{C} \sum_{i=1}^k d(\gamma(t_{i-1}),\gamma(t_i)),
$$
where the infimum is taken over all chains $C = \{ 0 = t_0,t_1,\ldots,t_n = t \}$ on 
$[0,t]$ with unbounded $k$.
A metric space is a {\it geodesic space}
if for any two points there exists at least one arc connecting them whose length 
equals to their distance.
Such an arc is called a {\it geodesic}.
More generally an arc $\gamma$ is {\it quasi geodesic} connecting $x,y \in X$ if there exist 
constants $a \ge 1,b>0$ so that
$$
	a^{-1} d(x,y) -b \le len(\gamma) \le a d(x,y) +b.
$$ 
  
\begin{prop}
	$(D,d_D)$ is
	\begin{itemize}
	\item[(i)]  a proper (i.e. any closed ball is compact) complete metric space and,
	\item[(ii)] a uniquely geodesic space.
	\end{itemize}
\end{prop}

\begin{proof}
	(i)\quad We denote $d_E$ be the Euclidean metric on $D$.
		Then the identity map $\id : (D,d_E)\ar (D,d)$ is continuous.
		In fact, 
		fix a point $x$ in $D$ and consider a sequence $\{y_i\}_i$ in $D$ converging to $x$.
		For each $i \in \N$, take
		$a_i,b_i \in \partial D$ so that four points $a_i,x,y_i,b_i$ are collinear.
		Then since $y_i \ar x$ ($i \ar \infty$), we have
		\[
			\frac{\|y_i-a_i\|\ \|x-b_i\|}{\|y_i-b_i\|\ \|x-a_i\|} \ar 1.
		\]
		This shows that $d(x,y_i) \ar 0$, hence $\id$ is continuous.
		Furthermore any closed ball in $(D,d)$ is an image of a compact set in $(D,d_E)$.
		In fact it is bounded closed set in $(D,d_E)$ since $D$ is bounded with respect to 
		the Euclidean metric and the identity map is continuous.
		Therefore any closed ball in $(D,d)$ is compact.
					
		By the properness of $(D,d)$,
		any Cauchy sequence $\{x_m\}_m$ in $(D,d)$
		has at least one converging subsequence in $D$
		since the Cauchy sequences are bounded.
		This implies that $\{x_m\}_m$ itself converges in $D$.
		
	(ii)\quad We can see that the Hilbert metric is a geodesic space by
		the following so-called {\it straightness property}.
		For any $x,y \in V$, $[x,y]$ denotes the Euclidean segment connecting $x$ and $y$. 
		
		If three points $x,y,z \in D$ are on the segment $[a,b]\ (a,b \in \partial D)$
		in this order, then $d(x,z) = d(x,y) + d(y,z)$.
		In fact, we have
		\begin{align*}
			d(x,y) + d(y,z) 
					   &= \frac{1}{2} \left( \log[a,x,y,b] + \log[a,y,z,b] \right)\\
					   &= \frac{1}{2} \log \left( 
					   	\frac{q(y-a)\ q(x-b)}{q(y-b)\ q(x-a)} \cdot 
						\frac{q(z-a)\ q(y-b)}{q(z-b)\ q(y-a)} \right)\\
					   &= \frac{1}{2} \log \left( 
					    \frac{q(z-a)\ q(x-b)}{q(z-b)\ q(x-a)} \right)\\
					   &= d(x,z).
		\end{align*}
		Thus the length of the segment $[x,z]$ realizes the metric $d(x,z)$.
		Furthermore since $D$ is strictly convex, geodesics are unique. 
\end{proof}

Let $(X,d)$ be a geodesic space.
For $x,y,p \in X$, we define the {\it Gromov product} $(x\vert y)_p$ 
of $x$ and $y$ with respect to $p$ by the equality
\[
	(x \vert y)_p = \frac{1}{2}\left(d(x,p) + d(y,p) - d(x,y)\right).
\]
Using this, the hyperbolicity in the sense of Gromov is defined as follows. 
For $\delta \ge 0$
the space $X$ is {\it $\delta$-hyperbolic} if 
\[
	(x \vert z)_p \ge \min\{ (x \vert y)_p, (y \vert z)_p \} -\delta 
\]
for all $x,y,z,p \in X$.
We say the space is simply {\it Gromov hyperbolic} if $X$ is $\delta$-hyperbolic for some
$\delta \ge 0$.

A {\it geodesic triangle} $T \subset X$ with vertices $x,y,z \in X$ is a union
of three geodesic curves with end points $x,y,z$.
We call these curves the {\it sides} of T.
A {\it triangle map} is a map $f : T \ar \R^2$ from geodesic triangle onto an Euclidean
triangle whose sides have the same length as corresponding sides of $T$,
and such that the restriction of $f$ to any one side is an isometry.
We always have triangle maps and they are unique up to isometry of $\R^2$
for a geodesic triangle.
A geodesic space is called a {\it CAT(0) space} if for any geodesic triangle $T$,
$d(x,y) \le |f(x) - f(y)|$ for all $x,y \in T$ whenever $f : T \ar \R^2$ is a triangle map.  

A metric space $(D,d_D)$ with the Hilbert metric is 
a CAT(0) and Gromov hyperbolic space since the region $D$ is an ellipsoid.
The former derived from a result given in \cite{eg} by Egloff.

\begin{thm}[Egloff]
	Let $H \subset \R^n$ be a convex open set with the Hilbert metric $d_H$.
	Then $(H,d_H)$ is a CAT(0) space if and only if $H$ is an ellipsoid.
\end{thm}

The latter owe to a result of Karlsson-Noskov \cite{KN}.

\begin{thm}[Karlsson-Noskov]
	Let $H \subset \R^n$ be a convex open set with the Hilbert metric $d_H$.
	If $H$ is an ellipsoid, then $(H,d_H)$ is a Gromov hyperbolic.
\end{thm}

The point of our definition of the Hilbert metric can be seen in the proof 
of the following proposition.

\begin{prop}\label{isom}
	Let $W$ be a Coxeter group with signature $(n-1,1)$.
	The normalized action of any $w \in W$ is an isometry on $(D,d_D)$.
\end{prop} 

\begin{proof}
	It suffices to show that the cross ratio 
	defining the Hilbert metric $d$ is invariant under 
	any normalized $B$-reflection $s_{\alpha}$ ($\alpha \in \Delta$).
	We take $x,y \in D$ arbitrary and let $a,b \in \partial D$ be the points 
	satisfying $d(x,y) = (1/2)\log[a,x,y,b]$.
	
	We check that $B$-reflection $s_{\alpha}$ preserves $q$.
	For any $v \in V$ and $\alpha \in \Delta$ we have
	\[
		q(s_{\alpha}(v))= q(v) - 4B(v,\alpha)^2 + 4B(v,\alpha)^2q(\alpha)
	                 	= q(v).
	\]
	This means that $[a,x,y,b] = [s_\alpha(a),s_\alpha(x),s_\alpha(y),s_\alpha(b)]$.
	
	Our remaining task is to show that $[s_\alpha(a),s_\alpha(x),s_\alpha(y),s_\alpha(b)]$ 
	does not vary under the normalization for $|*|_1$ in $Q_-^+$.
	This follows from Proposition \ref{1}.
	In fact, since $s_{\alpha}$ is linear, a segment is mapped to a segment.
	So the image $s_\alpha([a,b])$ coincides with $[s_{\alpha}(a),s_{\alpha}(b)]$.
	In particular four points $\{s_\alpha(a),s_\alpha(x),s_\alpha(y),s_\alpha(b)\}$ 
	are collinear.
	Furthermore $s_\alpha(x)$ and $s_\alpha(y)$ are in $Q_-^+$
	because the image of a segment in $Q_-^+$ by $s_\alpha$ does not include ${\bf 0}$.
	This means that $s_\alpha(a)-s_\alpha(b) \not\in Q$.
	At last, recall that for any $v \in Q_-^+$, $\widehat{v}$ lies on the ray 
	through ${\bf 0}$ and $v$.
	Therefore for each $z \in \{s_\alpha(a),s_\alpha(x),s_\alpha(y),s_\alpha(b)\}$,
	we have a ray through ${\bf 0}$ and $\widehat{z}$.
\end{proof}


\section{The properness of the normalized action}
We verify that the normalized action on $(D,d_D)$ is proper.
In general an isometric group action $G \curvearrowright X$ on a metric space $X$
is {\it proper} if for any compact set $F$ the set
\[
	\{g\in G \ \vert\ g(F) \cap F \neq \emptyset\} \subset G
\]
is finite.
We denote the action $G \curvearrowright X$ by $g.x$ for $g \in F$ and $x \in X$. 
If $X$ is locally compact and there exists a fundamental region $R$ 
(see Definition \ref{kihonryoiki})
then the action is proper.

\subsection{A fundamental region and a Dirichret region}
We define two open sets (with respect to the subspace topology of $V_1$)
\[
	K := \{v \in D \ \vert\ \forall \alpha \in \Delta, B(\alpha,v) < 0\}
	\quad \quad \text{and} \quad \quad 
	K' := K \cap D'.
\]
For $\alpha \in \Delta$ we set 
$P_\alpha = \{v \in V_1\ \vert\ \text{$\alpha$-th coordinate\ of\ $v$\ is\ $0$}\}$
and $H_\alpha = \{v \in V_1\ \vert\ B(v,\alpha) = 0\}$.
We define 
\[
	\PP = \{v \in V_1 \ \vert\ \forall \alpha \in \Delta, B(\alpha,v) < 0\}
	\quad \quad \text{and}\quad \quad
	\PP' = \PP \cap \inte(\conv(\widehat{\Delta})).
\] 
Then clearly $K = \PP \cap D$. 
Moreover, we will see that $K' = \PP' \cap D$ (Lemma \ref{katachi}). 
Since $\PP$ (resp. $\PP'$) is bounded by finitely many $n-1$ dimensional subspaces 
$\{H_\alpha\ \vert\ \alpha \in \Delta\}$ 
(resp. $\{H_\alpha\ \vert\ \alpha \in \Delta\}$ and $\{ P_\alpha \ \vert\ \alpha \in \Delta\}$),
actually $\overline{\PP}$ (resp. $\overline{\PP'}$) is a polyhedron.
In general, $\PP$ is not a simplex.
The following example of $W$ such that $\PP$ is not a simplex is given by Yohei Komori.
\[
	W = \langle s_1,\ldots,s_5\ \vert\ s_i^2, (s_{i-1}s_i)^4\rangle,
\]
where $i= 1,\ldots,5$ and $s_0 = s_5$.
In fact the Coxeter graph of this does not appear 
in the list given by Schlettwein \cite{schlettwein}.

\begin{dfn}\label{kihonryoiki}
	We assume that a group $G$ acts on a metric space $X$ isometrically.  
	We denote the action by $g.x$ for $g \in G$ and $x \in X$.
	Then an open set $A \subset X$ is
	\begin{itemize}
	\item 
	a {\it fundamental region} if  
	$\overline{G.A} = X$ and $g.A \cap A = \emptyset$ for any $g \in G$ where
	$\overline{G.A}$ is the topological closure of $G.A$;
	\item
	the {\it Dirichlet region} at $o \in A$ if $A$ equals to the set
	\[
	\{x \in D\ \vert\ 
		d(o,x) < d(o,w\cdot x)\ \text{for}\ w \in W \setminus \{\id\}\}.
	\]
	\end{itemize}
\end{dfn}

We will show that $K$ (resp. $K'$) is the Dirichlet region at any $x \in K$ 
hence a fundamental region for the (resp. restricted) normalized action of $W$ on $D$.

\begin{rem}\label{bb}
	By \cite[Proposition 4.2.5]{bb},
	for $w \in W$ and $s_\alpha \in S$ if $|sw| > |w|$ then 
	all coordinates of $w^{-1}(\alpha)$ are non-negative.
\end{rem}

\begin{prop}\label{eigen}
	For any $z \in K$, we have the followings.
	\begin{itemize}
	\item[(i)] For any $w \in W \setminus \{id\}$, 
			   there exists $\alpha \in \Delta$ so that $B(w\cdot z,\alpha) > 0$:
	\item[(ii)] For any $w \in W$, $|w(z)|_1 > 0$.
				Moreover, if $z \in \inte(\conv(\widehat{\Delta}))$ 
				then all coordinates of $w(z)$ are positive.
	\end{itemize}
\end{prop}

\begin{proof}
	We prove (i) and (ii) at the same time by the induction for the word length.

	In the case $|w| = 1$, there exists $\alpha \in \Delta$ so that $w = s_\alpha$.
	Then we have
	\[
		|s_\alpha (z)|_1 = |z|_1 -2B(z,\alpha)|\alpha|_1 > 0.
	\]
	Therefore 
	\[
		B(s_\alpha \cdot z,\alpha) = \frac{B(z,-\alpha)}{|s_\alpha (z)|_1} > 0,
	\]
	since $s_\alpha(\alpha) = -\alpha$.

	If $|w| > 1$,
	there exist $\alpha \in \Delta$ and $w' \in W$
	satisfying $w = s_\alpha w'$.
	In particular $|w'| = |w| -1$.
	We have $|w'(z)|_1 > 0$ by the assumption of the induction.
	From Remark \ref{bb} we have $v_\beta \ge 0$ if 
	$w'^{-1}(\alpha) = \sum_{\beta \in \Delta} v_\beta \beta$.
	Then we see that 
	\begin{align}\label{hutoushiki}
		|w(z)|_1 &= |s_\alpha(w'(z))|_1 = |w'(z)|_1 - 2B(z,w'^{-1}(\alpha))|\alpha|_1 \notag \\
				 &= |w'(o)|_1 -2 |\alpha|_1\sum_{\beta \in \Delta} v_\beta B(z,\beta) > 0,
	\end{align}
	because $z \in K$.
	This shows (ii).
	In addition, we have
	\begin{align*}
		B(w \cdot z,\alpha) = \frac{-B(w'(z),\alpha)}{|s_\alpha(w'(z))|_1}
							= \frac{-\sum_{\beta \in \Delta} v_\beta B(z,\beta)}
							       {|s_\alpha(w'(o))|_1}
							> 0.
	\end{align*}
	Hence we have (i).	
\end{proof}

This lemma ensures that $K$ and $K'$ are not empty.

\begin{lem}\label{katachi}
	We have the following:
	\begin{itemize}
	\item[(i)]
	$K' = K \cap \inte(\conv(\widehat{\Delta})) = \PP' \cap D.$
	\item[(ii)]
	$K'$ (hence $K$) is not empty.
	\end{itemize}
\end{lem}

\begin{proof}
	(i)\ 
	Recall that $R = D \setminus \conv(\widehat{\Delta})$.
	We set $K_{int} = K \cap \inte(\conv(\widehat{\Delta}))$.
	Then clearly $K' \subset K_{int}$.
	To see the inverse inclusion, 
	it suffices to show that $w \cdot R \cap K_{int} = \emptyset$ for any $w \in W$.
	Take $x \in K_{int}$ arbitrarily.
	Then $x_\alpha > 0$ for any $\alpha \in \Delta$ 
	if we write $x = \sum_{\alpha \in \Delta} x_\alpha \alpha$.
	Now we assume that $w \cdot x \in R$ then there exists $\alpha \in \Delta$
	such that $\alpha$-th coordinate of $w \cdot x$ is non-positive. 
	This contradicts to the latter claim of Proposition \ref{eigen} (ii).
	
	(ii)\ 
	Let $o$ be the normalized eigenvector for the negative eigenvalue $-\lambda$ of $B$.
	Then all coordinates of $o$ are positive by the definition and Lemma \ref{heimen}.
	For any $\alpha \in \Delta$, we have 
	\[
	B(o,\alpha) = -\lambda(o,\alpha) <0.
	\]
	Thus $o \in K$.
	Furthermore by the same argument as the proof of (i), we also have $o \in K'$
	since all coordinates of $o$ are positive.
\end{proof}

As a consequence of Proposition \ref{eigen}, we have the following.

\begin{lem}\label{kihon2}
	For any $w \in W \setminus\{\id\}$, we have $w \cdot K \cap K = \emptyset$.
\end{lem}

\begin{lem}\label{majiwari}
	For any $x \in K$ and $\xi \in \partial D$ (or in $\partial D' \setminus D$)
	the Euclidean segment $[x,\xi]$ joining $x$ and $\xi$ is not contained in 
	any hyperplane $w\cdot H_\alpha$ ($w \in W$, $\alpha \in \Delta$).
\end{lem}

\begin{lem}\label{kihon}
	For any $x \in K$, $K$ is the Dirichlet region at $x$.
\end{lem}

\begin{proof}
	We assume that there exists a point $y$ in the Dirichlet region at $x$ such that
	$y \not\in K$.
	Then by the definition of $K$ we have $\alpha \in \Delta$ satisfying $B(\alpha,y) \ge 0$.
	If $B(\alpha,y) = 0$ then $s_\alpha(y) = y$ and hence 
	$d(x,y) = d(x,s_\alpha \cdot y)$ which is a contradiction.
	For the other case $B(\alpha,y) > 0$, then the Euclidean segment $[o,x]$ joining $x$ and $y$
	intersects with $H_\alpha$.
	Let $z$ be the intersection point.
	Since $z$ fixed by $s_\alpha$, $d(s_\alpha \cdot y,z) 
	= d(s_\alpha \cdot y, s_\alpha \cdot z ) = d(o,z)$.
	Then we have $d(x,y) = d(x,z) + d(z,y) = d(x,z) + d(z, s_\alpha \cdot y)$,
	hence $d(x,y) \ge d(x,s_\alpha \cdot y)$ by the triangle inequality.
	This contradicts to the hypothesis that $y$ is in the Dirichlet region at $x$.
	
	For the inverse, assume that $y \in K$ is not in the Dirichlet region at $x$.
	By Lemma \ref{kihon2} there exists an element $w \in W \setminus\{\id\}$ that attains 
	$\min_{w \in W \setminus \{\id\}} d(y,w\cdot x)$ and satisfies $w \cdot x \not\in K$.
	Consequently there exists $\alpha \in \Delta$ such that the Euclidean segment $[w \cdot x,y]$ 
	joining $w \cdot x$ and $y$ intersects with $H_\alpha$.
	The intersection point $z$ is fixed by $s_\alpha$ hence $d(s_\alpha \cdot x,z) = d(x,z)$.
	The uniqueness of the geodesic between $y$ and $(s_\alpha w) \cdot x$ gives 
	$d(y,(s_\alpha w) \cdot x) < d(y,z)+d(z,(s_\alpha w) \cdot x) = d(y,w\cdot x)$.
	This contradicts to the minimality of $d(y,w\cdot x)$.
\end{proof}

Lemma \ref{kihon} shows also that $K$ is connected.
In fact, assume that $K$ has more than two components.
Then we can decompose $K$ into $K_1 \sqcup K_2$ and assume that $o \in K_1$.
Take $v \in K_2$ and consider the geodesic $\gamma$ from $o$ to $v$.
Then $\gamma$ should pass through at least one hyperplane $H_\alpha$.
Let $u$ be an intersection point.
Since $u \in H_\alpha$, we have $w \cdot u = u$.
Now we see that 
$d(o, s_\alpha \cdot v) < d(o, u) + d(s_\alpha \cdot u, s_\alpha \cdot v) 
= d(o,u) + d(u,v) = d(o,v)$.
This contradicts to Lemma \ref{kihon}.
  
\begin{prop}
	$K$ is a fundamental region for the normalized action. 
\end{prop}

\begin{proof}
	Take $y \in D$ arbitrary.		
	Let $w \cdot o$ be the nearest orbit of $o$ from $y$.
	Then we see that $w^{-1} \cdot y \in \overline{K}$ by Lemma \ref{kihon}. 
	The second assertion of the definition of the fundamental region
	is Lemma \ref{kihon2}.
\end{proof}

The following corollary is originally proved by Floyd \cite[Lemma in p.213]{Floyd} for 
geometrically finite Kleinnian groups without parabolic elements.
Here we assume the Coxeter groups $W$ acts cocompactly,
i.e., the quotient space of $D$ by the normalized action is compact.
In that case, polytope $\overline{K}$ is contained in $D$.

\begin{cor}\label{qi}
	Let $o$ be the normalized eigenvector for the negative eigenvalue of $B$.
	If the fundamental region $K$ is bounded, 
	then there are constants $k,k' > 0$ so that 
	$k|w| \le d(w\cdot o, o) \le k'|w|$ for all $w \in W$.
\end{cor}

\begin{dfn}\label{seq}
	Let $(W,S)$ be a Coxeter system.
	\begin{itemize}
	\item
	We call a sequence $\{w_k\}_k$ in $W$ a {\it short sequence}
	if for each $n \in \N$ there exists $s \in S$ such that 
	$w_{k+1} = sw_k$ and $|w_k| = k$.
	\item
	For a sequence $\{w_k\}_k$ in $W$, a path in $V_1$
	is a {\it sequence path} for $\{w_k\}_k$ if the path is given by 
	connecting Euclidean segments $[w_k \cdot o,w_{k+1} \cdot o]$ for all $k \in \N$.
	\end{itemize}
\end{dfn}

\begin{rem}
	A {\it reflection} in $W$ is an element of the form $w s w^{-1}$ for $s \in S$ and $w \in W$.
	We see that $w \cdot H_\alpha = H_{w \cdot \widehat{\alpha}}$.
	We remark that each reflection $w s_\alpha w^{-1}$ corresponds to 
	the normalized $B$-reflection with respect to $H_{w \cdot \widehat{\alpha}}$. 
	We say that the normalized action of $w s_\alpha w^{-1}$ 
	to be the reflection for $w \cdot H_\alpha$.
\end{rem}

\begin{prop}\label{shortest}
	Suppose that $W$ acts on $D$ cocompactly.
	For any $\xi \in \Lambda(W)$ there exists a short sequence $\{w_k\}_k$
	so that $w_k \cdot o$ converges to $\xi$.
	Furthermore the sequence path for $\{w_k\}_k$ lies in $c$-neighborhood of a segment 
	$[o,\xi]$ connecting $o$ and $\xi$ for some $c > 0$
	with respect to the Hilbert metric.
\end{prop}

\begin{proof}
	A Euclidean segment $\gamma = [o,\xi]$ is a geodesic ray with respect to the Hilbert metric.
	The segment $\gamma$ intersects with infinitely many hyperplanes 
	$\{w_k \cdot H_{\alpha_k}\}$ 
	($\alpha_k \in \Delta,\ w_k \in W$ for $k \in \N$) transversely
	since it is not contained any hyperplane $w_k \cdot H_{\alpha_k}$
	by Lemma \ref{majiwari}.
	We notice that $\gamma$ pass through each $\{w_k \cdot H_{\alpha_k}\}$ only once
	because the Euclidean straight line cannot pass through any hyperplane twice.
	If $\gamma$ intersects with some hyperplanes at the same point $x$, then
	by perturbing subpath of $\gamma$ in the $\epsilon$ ball $B(x,\epsilon)$ centered at $x$
	we have a quasi geodesic ray $\gamma'$ toward $\xi$ 
	which is in $\epsilon$ neighborhood of $\gamma$ for sufficiently small $\epsilon >0$.
	Then $\gamma'$ intersects with the hyperplanes passing through $x$ only once.
	In particular, $\gamma'$ intersects with distinct hyperplanes at distinct points.
	
	We renumber the hyperplanes $\{w_k \cdot H_{\alpha_k}\}$ with which $\gamma'$ intersects
	so that 
	if $\gamma'$ intersects with some hyperplanes $w_k \cdot H_{\alpha_k}$, 
	$w_{k'} \cdot H_{\alpha_{k'}}$ at $\gamma'(t)$, $\gamma'(t')$ respectively for $t<t'$,
	then we have $k<k'$.
	Thus we have a sequence $\{w_k\}_k$.
	Considering the $B$-reflection $s_{\alpha_i}$ with respect to $H_{\alpha_i}$ 
	for each $i \in \N$,
	we see that $w_k = s_{\alpha_k}s_{\alpha_{k-1}}\cdots s_{\alpha_1}$ for each $k \in \N$.
	Let $w_0 = id$ and $r_i = w_{i-1}s_{\alpha_i}w^{-1}_{i-1}$ for $i \in \N$.
	Then $w_k = r_k\cdots r_1$.
	Now each $r_i$ is the reflection for $w_i \cdot H_{\alpha_i}$ for all $i$.
	Since $\gamma'$ meets each hyperplane $w_i \cdot H_{\alpha_i}$ ($i \in \N$) only once,
	in the sequence $\{r_k,\ldots,r_1\}$ no reflection occurs more than once
	for all $k \in \N$.
	This shows that the word $s_{\alpha_k}s_{\alpha_{k-1}}\cdots s_{\alpha_1}$ 
	is a geodesic for $w_k$ (\cite[Corollary 3.2.7]{davis}).
	Therefore, the sequence $\{w_k\}_k$ is a short sequence.
		
	Furthermore by the construction, we see that 
	the sequence path for $\{w_k\}_k$ is included in $c$-neighborhood of $\gamma$,
	where $c$ equals to the diameter of $K$.
\end{proof}

\subsection{Three cases}
We consider the normalized action 
by dividing it into the following three cases:
cocompact, convex cocompact, with cusps.
We recall that $\conv(\widehat{\Delta})$ is a simplex. 
It can happen three distinct situations due to the bilinear form $B$;
\begin{itemize}
	\item[(i)] the region $D \cup \partial D$ is included in $\inte(\conv(\widehat{\Delta}))$;
	\item[(ii)] there exist some $n'$ ($<n$) dimensional faces of $\conv(\widehat{\Delta})$ 
	which are tangent to the boundary $\partial D$;
	\item[(iii)] 
	$D \cup \partial D \not\subset \inte(\conv(\widehat{\Delta}))$ and 
	no faces of $\conv(\widehat{\Delta})$ tangent to $\partial D$.
\end{itemize}
We argue the cases (i) and (iii) simultaneously.
For the case (ii), we can not apply the same argument as (i) and (iii).
The most general case will be discussed in Section 4.2. 

\begin{rem}\label{bubungyoretsu}
	By \cite[Corollary 2.2]{Heamers}, we see that a Coxeter subsystem $(W',S')$ 
	satisfying $S' \subset S$ is either of type $(|S'|-1,1)$ or $(|S'|-1,0)$ or positive definite.
	Let $B'$ be the bilinear form corresponding to $(W',S')$.
	If $B'$ has the signature $(|S'|-1,1)$ (resp. $(|S'|-1,0)$),
	then by the same argument as Lemma \ref{heimen}, we have an eigenvector 
	$o' \in \mathrm{span}(\Delta')$ of the negative 
	eigenvector (resp. $0$ eigenvalue) such that all coordinates of $o'$ for $\Delta'$  
	are positive where $\mathrm{span(\Delta')}$ 
	denotes the subspace spanned by $\Delta'$.
	This shows that $Q' = \{v \in \mathrm{span(\Delta')}\ \vert\ B'(v,v) = 0\}$ should intersect
	with $\conv(\widehat{\Delta'})$. 
	Since the Gram matrixof $B'$ is a principal submatrix of the Gram matrixof $B$,
	we see that $\partial D \cap \conv(\widehat{\Delta'}) = Q' \cap \conv(\widehat{\Delta'})$.
	Thus we have the followings:
	\begin{itemize}
	\item[(1)] $B'$ has the signature $(|S'|-1,1)$ if and only if 
	$D \cap \conv(\Delta') \not= \emptyset$; 
	\item[(2)] $B'$ has the signature $(|S'|-1,0)$ if and only if 
	$\partial D \cap \conv(\Delta') = Q' \cap \conv(\widehat{\Delta'})$, which is a singleton;
	\item[(3)] $B'$ is positive definite if and only if
	$(D \cup \partial D) \cap \conv(\widehat{\Delta'}) = \emptyset$.
	\end{itemize}
	If $B'$ has the signature $(|S'|-1,1)$ then $H_\alpha$ for $\alpha \in \Delta'$
	intersects with $D \cap \conv(\Delta')$.
	In fact if not, then $D \cap \conv(\widehat{\Delta'})$ is not preserved by 
	$s_\alpha$ for $\alpha \in \Delta'$.
	Moreover, by the compactness of $Q$, $Q' \cap V_0 = {\bf 0}$ 
	for any Coxeter subsystem $(W',S')$. 
\end{rem}

We say a Coxeter system of rank $n$ is {\it affine} if its associating bi-linear form $B$
has the signature $(n-1,0)$. 
Fixing a generating set $S$ we simply say Coxeter group $W$ is affine if the Coxeter system 
$(W,S)$ is affine.
An affine Coxeter group is of infinite order and 
its limit set is a singleton (\cite[Corollary 2.15]{hlr}).
We notice that for any affine Coxeter group
if its rank is more than $2$ then there are no simple roots $\alpha,\beta \in \Delta$
with $B(\alpha,\beta) \le -1$ if $B$ is irreducible. 
In fact if $B(\alpha,\beta) \le -1$ then the subgroup generated by 
$s_\alpha, s_\beta$ is of infinite order hence $E \cap \conv(\{\alpha,\beta\}) \neq \emptyset$.
This implies that $E \subset \conv(\{\alpha,\beta\})$ since $E$ is a singleton.
Hence $B(\alpha,\beta) < -1$ does not occur because if so then 
the limit set of the subgroup generated by $s_\alpha, s_\beta$ consists of two points.
For the case where $B(\alpha,\beta) = -1$, let $x$ be the limit point, i.e, $E =\{x\}$.
Since $x \in \conv(\{\alpha,\beta\})$,
for any $\gamma \in \Delta \setminus \{\alpha,\beta\}$, the $\gamma$-th coordinate
of $x$ equals to $0$.
For the $\alpha$-th coordinate and the $\beta$-th coordinate of $x$ are not $0$.
Since $B$ is irreducible, $B(x,\gamma) \not= 0$ for $\gamma \in \Delta\setminus \{\alpha,\beta\}$.
This shows that $s_\gamma \cdot x \not= x$.
However since $s_\gamma \cdot x$ is in $E$, we have a contradiction.

\begin{rem}\label{combi}
	We remark that for $x \in \partial D$ and $\alpha \in \Delta$ we have 
	$\{x,s_\alpha\cdot x\} = L(\widehat{\alpha},x) \cap \partial D$ 
	where $L(\widehat{\alpha},x)$ is the Euclidean line 
	passing through $\widehat{\alpha}$ and $x$.
	This is because that $s_\alpha\cdot x$ is a linear combination of $\widehat{\alpha}$ and $x$,
	and $\partial D$ is preserved by $s_\alpha$.
\end{rem}

\begin{prop}\label{bunrui1}
	Assume that $(W,S)$ is Coxeter system of type $(n-1,1)$.
	\begin{itemize}
	\item[(a)] 
	The case (i) happens if and only if every Coxeter subgroup of $W$ of rank $n-1$ 
	generated by a subset of $S$ is finite.
	\item[(b)]
	The case (ii) happens if and only if 
	there exists a rank $n'$ ($<n$) affine Coxeter subgroup of $W$ 
	generated by a subset of $S$.
	\item[(c)]
	The case (iii) happens if and only if every Coxeter subgroup of $W$ of rank $n'$ ($<n$)
	generated by a subset of $S$ is of type $(n'-1,1)$ or $(n',0)$.
	\end{itemize}
\end{prop}

\begin{proof}
	Note that for $\Delta' \subset \Delta$ we can restrict the bi-linear form $B$ to $\Delta'$.
	We denote such a bi-linear form as $B'$,
	namely, the Gram matrixwith respect to $B'$ is a principal submatrix of the 
	Gram matrixwith respect to $B$.
	
	(a)
	Let $W'$ be a Coxeter subgroup of rank $n-1$ and let $B'$ be the bilinear form for $W'$.
	Recall a classical result that $W'$ is finite if and only if $B'$ is positive definite
	(see \cite[Theorem 6.4]{Hum90}).
	This is equivalent to that $\widehat{Q}$ does not intersect with $\conv(\widehat{\Delta'})$. 
	
	(b)
	Assume that there exists an affine rank $n'$ ($<n$) Coxeter subgroup $W'$ of $W$
	generated by a subset $S'$ of $S$ which is minimal.
	Let $\Delta'$ and $B'$ be the subset of $\Delta$ and 
	the bilinear form corresponding to $W'$ respectively.
	As we have mentioned before, 
	the normalized limit set $\widehat{\Lambda(W')}$ is a singleton $\{\xi\}$
	and it equals to $\widehat{Q'} \subset \widehat{Q}$.
	This shows that an $n'$ dimensional face $\conv(\widehat{\Delta'})$ is tangent to $\widehat{Q}$.
		
	For the converse we assume that an $n' (< n)$ dimensional face $\conv(\widehat{\Delta'})$ is
	tangent to $\widehat{Q}$ for some $\Delta' \subset \Delta$.
	We also assume that face is minimal.
	Let $S' \subset S$ and $W'$ be 
	the set of simple $B'$-reflection corresponding to $\Delta'$ and the Coxeter subgroup of $W$
	generated by $S'$ respectively.
	Then for the corresponding bilinear form $B'$ the set 
	$\widehat{Q'}$ consists of one point $v$.
	Then $s_\alpha \cdot v = v$ for any $\alpha \in \Delta'$ by Remark \ref{combi}.
	Therefore $B'(v,\alpha) = 0$ for any $\alpha \in \Delta'$ and hence $v$ is an eigenvector of 
	$0$ eigenvalue of $B'$.
	This means that $B'$ has the signature $(n'-1,0)$ and hence $W'$ is affine.
		
	(c)
	If every infinite rank $n'$ ($<n$) Coxeter subgroup of $W$
	generated by a proper subset $S'$ of $S$ is of type $(n'-1,1)$ or positive definite
	then $\widehat{Q'}$ is either an ellipsoid or empty. 
	Obviously $\widehat{Q'} \subset \partial D$ we have the case (iii).
	If the case (iii) happens, 
	then $\partial D$ should intersect with a face of $\conv(\widehat{\Delta})$.
	Let $\conv(\widehat{\Delta'})$ be such a face and let $B'$ be 
	the bilinear form corresponding to $\Delta'$.
	Then there exists $v \in D \cap \conv(\widehat{\Delta'})$.
	Since $B(v,v) = B'(v,v) <0$,
	we see that $B'$ has the signature $(|\Delta'|-1,1)$.
	For $\Delta'' \subset \Delta$ 
	if $\conv(\widehat{\Delta''})$ does not intersect with $D\cup \partial D$  
	then there are no elements $v \in \conv(\widehat{\Delta''})$ such that $B''(v,v) = 0$ 
	where $B''$ is the bilinear form corresponding to $\Delta''$.
	This is because that $B''$ is a principal submatrix of $B$.
	Thus $B''$ is positive definite.
\end{proof}
	
\begin{rem}\label{core}
	For $v = \sum_{\alpha\in \Delta}v_\alpha \alpha$ we have 
	\[
		q(v) = \sum_{\alpha \in \Delta} v_\alpha B(v,\alpha).
	\]
	From this,
	if there exists $v \in \PP$ such that $q(v) \ge 0$ then $v_\alpha <0$ for some 
	$\alpha \in \Delta$.
\end{rem}

\begin{prop}\label{bunrui}
	For each case, we have followings:
	\begin{itemize}
	\item[(a)] The case $(i)$ $\iff$ $\overline{\PP} = \overline{\PP'} \subset D$;
	\item[(b)] the case $(ii)$ $\iff$ $\overline{\PP'}$ has some
	vertices in $\partial D$;
	\item[(c)] the case $(iii)$ $\iff$ $\overline{\PP}\not= \overline{\PP'}$ 
	and no vertices of $\overline{\PP'}$ belong to $\partial D$.
	\end{itemize}
\end{prop}

\begin{proof}
	(a)
	Remark \ref{core} shows that if (i) then $\PP \subset D$.
	Moreover, if there exists a vertex $v$ of $\overline{\PP}$ such that $q(v) \ge 0$ then 
	$v_\alpha \le 0$ for some $\alpha \in \Delta$ by the definition.
	This implies that $v \in \conv(\widehat{\Delta} \setminus \{\alpha\})$ and hence
	$\partial D$ must intersect with the face $\conv(\widehat{\Delta} \setminus \{\alpha\})$.
	This is a contradiction.
	Thus $\overline{\PP} \subset D \subset \conv(\widehat{\Delta})$ and $\PP = \PP'$.
	
	Conversely, we assume that $\overline{\PP} \subset D$.
	For any $\Delta' \subset \Delta$, let $S'$ be the subset of $S$ corresponding to $\Delta'$.
	Let $B'$ be the bilinear form corresponding to $S'$.
	If $\conv(\widehat{\Delta'}) \cap \partial D \not= \emptyset$
	then the Coxeter subgroup $W'$ generated by $S'$ is infinite. 
	In particular, any Coxeter element in $W'$ has infinite order.
	Moreover for $\alpha \in \Delta'$, $H_\alpha$ should intersect with 
	$D \cap \conv(\widehat{\Delta'})$.
	Let $v$ be a point of $\overline{\PP}$ such that 
	$v \in \bigcap_{\alpha \in \Delta'} H_\alpha$.
	Now $v \in D$ by our assumption.
	Then since $v$ is fixed by any element in $W'$, we have an accumulation point in $D$.
	This contradicts to the properness of the normalized action of $W$ on $D$.

	(b)
	We have a face of $\conv(\widehat{\Delta})$ which is tangent to
	$\partial D = \widehat{Q}$.
	Let $\Delta'$ be the minimal subset of $\Delta$ such that 
	$\conv(\widehat{\Delta'})$ is tangent to $\partial D$ and let $v$ be the point of tangency.
	We set $S' = \{s_\alpha\ \vert\ \alpha \in \Delta'\}$.
	Then for any $\alpha \in \Delta'$, we have $B(v,\alpha) = 0$ since $v = s_\alpha \cdot v$ 
	by Remark \ref{combi}.
	Thus $v \in \bigcap_{\alpha \in \Delta'} H_\alpha$.
	Furthermore since $v \in \conv(\widehat{\Delta'})$ we have $v_{\alpha} =0$ for 
	$\alpha \in \Delta \setminus \Delta'$ if we write 
	$v = \sum_{\alpha \in \Delta} v_{\alpha} \widehat{\alpha}$.
	Consequently
	\begin{align}\label{choten}
		\{v\} = \bigcap_{\alpha \in \Delta'} H_\alpha \cap 
		\bigcap_{\beta \in \Delta \setminus \Delta'} P_\beta.
	\end{align}
	This shows that $v$ is a vertex of $\overline{\PP'}$.
	
	Conversely, we assume that there exists a vertex $v$ of $\overline{\PP'}$ on $\partial D$
	satisfying \eqref{choten} for some $\Delta' \subset \Delta$.
	Then $v \in \mathrm{span}(\Delta')$ where $\mathrm{span}(\Delta')$ is the subspace of $V$
	spanned by $\Delta'$.
	In addition, since $v \in \bigcap_{\alpha \in \Delta'} H_\alpha$, we see that
	$v$ is an eigenvector of the Gram matrixfor $\Delta'$ 
	corresponding to the $0$ eigenvalue.
	Thus $\conv(\widehat{\Delta'})$ is tangent to $\partial D$.

	(c)
	We assume that $\conv(\widehat{\Delta'}) \cap D \not= \emptyset$ for some
	$\Delta' \subset \Delta$.
	Obviously $\PP \neq \PP'$.
	Moreover, by Remark \ref{core} we see that every vertex of $\overline{\PP'}$ belongs
	to $\partial D$ or $D$.
	However if there exists a vertex lying on $\partial D$ then the case (ii) happens 
	by the proof of (b).
	Thus all vertices of $\overline{\PP'}$ are in $D$. 
	The converse is clear by (a) and  (b).
\end{proof}

From Proposition \ref{bunrui} we deduce that the fundamental region $K$ (resp.$K'$) 
is bounded if the case (i) (resp. the case (iii)) happens.
If $\overline{K'}$ is not compact, 
then $\partial D$ must be tangent to some faces of $\conv(\widehat{\Delta})$.
In this case $K'$ has some cusps at points of tangency of $\partial D$.
This happens if and only if (ii).
Because of this we call each cases as follows:
The normalized action of $W$ on $D$ is 
\begin{itemize}
	\item 
	{\it cocompact} if the case (i) happens;
	\item
	{\it with cusps} if the case (ii) happens;
	\item
	{\it convex cocompact} if the case (iii) happens.
\end{itemize}

In the case (ii) the {\it rank} of cusp $v$ is the minimal rank of 
the affine Coxeter subgroup generated by a subset of $S$ 
which fixes $v$.
Note that we can find easily that there exist Coxeter groups corresponding to 
each cases (i), (ii) and (iii).
Thus all the possibilities may happen.

\begin{exa}
	We see that classical hyperbolic Coxeter groups are in the case (i).
	For the case (iii) one of the simplest example is a triangle group
	$W = \langle s_1,s_2,s_3\ \vert\ s_i^2\ (i=1,2,3)\rangle$ with bi-linear form satisfying
	$B(\alpha_i,\alpha_j) < -1$ for $i\neq j$.
	At last it is in the case (ii) that
	$W = \langle s_1,s_2,s_3,s_4\ \vert\ 
	s_i^2,(s_1s_2)^6,(s_1s_3)^3,(s_js_k)^2 (j\neq k \in \{2,3,4\})\rangle$
	with the matrix $(B(\alpha_i,\alpha_j))_{i,j}$ equals to
	\[
		\begin{bmatrix}
		1 & -\frac{\sqrt{3}}{2} & -\frac{1}{2} & T \\
		-\frac{\sqrt{3}}{2} & 1 & 0 & 0 \\
		-\frac{1}{2} & 0 & 1 & 0 \\
		T & 0 & 0 & 1 \\
		\end{bmatrix}
	\]
	where $T < -1$.
	In fact $W$ is with signature $(3,1)$ although a subgroup generated by 
	$\{s_1,s_2,s_3\}$ is with signature $(2,0)$.
\end{exa}


\subsection{The limit set and the set of accumulation points of roots}\label{E}

\begin{dfn}
	For a Coxeter system $(W,S)$ of type $(n-1,1)$, 
	let $o$ be the normalized eigenvector corresponding to the negative eigenvalue
	of the corresponding Gram matrix.
	The {\it limit set} $\Lambda_B(W)$ of $W$ with respect to $B$
	is the set of accumulation points of the orbit of $o$ 
	by the normalized action of $W$ on $D$ in the Euclidean topology. 
	The limit set depends on the Gram matrix$B$.
	If $B$ is understood, then we simply denote the limit set by $\Lambda(W)$.
\end{dfn}

Now, we claim $\Lambda(W) = E$.
Before proving this, we need to confirm the definition of the limit set is 
independent of the choice of the base point $o$.

\begin{lem}\label{limitdef}
	Let $\{x_k\}_k$ and $\{y_k\}_k$ be two sequences in $D \cap \inte(\conv(\widehat{\Delta}))$ 
	converging to the points $x$ and $y$ in $\partial D$ with respect to the Euclidean metric.
	If there exist a constant $C$ so that $d(x_k, y_k) \le C$ for all $k \in \N$ then $x=y$.
\end{lem}

\begin{proof}
	Recall that $q(x-y) = 0$ if and only if $x=y$ for $x,y \in \partial D$.
	Let $\{a_k\}_k$ and $\{b_k\}_k$ be two sequences in $\partial D$
	associating with $\{x_k\}_k$ and $\{y_k\}_k$ so that
	\[
		d(x_k,y_k) = \frac{1}{2} \log [a_k,x_k,y_k,b_k] 
					 = \frac{1}{2} \log \left(
					   \frac{q(y_k-a_k)\cdot q(x_k-b_k)}{q(y_k-b_k)\cdot q(x_k-a_k)}
					   \right)
	\]
	for all $k \in \N$.
	Now there exists a constant $C' > 0$ so that for any $z,z' \in D\cup \partial D$,
	$q(z' - z) \le C'$ since $D \cup \partial D$ is compact.
	Then we have 
	\[
		q(y_k-a_k)\cdot q(x_k-b_k)
			\le e^{2C}q(y_k-b_k)\cdot q(x_k-a_k)
			\le e^{2C}C'q(x_k-a_k).
	\]
	We have $a_k \ar x$ since $x_k \ar x \in \partial D \subset Q$.
	Hence the right hand side of the inequality above tends to $0$.
	Hence $q(y_k-a_k)$ or $q(x_k-b_k)$ converges to $0$.
	If $q(y_k-a_k)$ tends $0$ then $y_k \ar x $ since $\{x_k\}_k$ and $\{a_k\}_k$ 
	converge to the same point $x$. 
	For the other case, we also have $y_k \ar x $
	since $y_k$ is on the segment joining $a_k$ and $b_k$ for all $k \in \N$.
\end{proof}

Theorem \ref{main2} immediately follows form the next proposition. 

\begin{prop}\label{icchi}
	Let $\{w_n\}_n$ be a sequence of elements in $W$.  
	For any $\delta \in \Delta$ and $y \in D$,
	$w_n \cdot \widehat{\delta} \ar z \in \partial D$ if and only if 
	$w_n \cdot y \ar z \in \partial D$.
\end{prop}

\begin{proof}
	It suffices to show that in the case $y = o$ where $o$ is the 
	normalized negative eigenvector of $B$
	from Lemma \ref{limitdef}. 
	By \cite[Theorem 2.7]{hlr},
	we have that for any injective sequence $\{w_n\}_n$ in $W$ and $\alpha \in \Delta$,
	$||w_n(\alpha)|_1| \ar \infty$.
	This implies $||w_n(\widehat{\alpha})|_1| \ar \infty$.
	On the other hand by Proposition \ref{kihon},
	the normalized action is proper.
	This means $w_n \cdot o$ tends to $\partial D$,
	hence $q(w_n \cdot o) \ar 0$,
	equivalently $|w_n(o)|_1 \ar \infty$.
	
	Since $B(w(p),w(p')) = B(p,p')$ for any $p,p' \in V$ and $w \in W$, 
	it holds that
	\begin{align*}
		B(w_n \cdot \widehat{\alpha},w_n \cdot o) 
			&= B\left(\frac{w_n(\widehat{\alpha)}}{|w_n(\widehat{\alpha})|_1},
				\frac{w_n(o)}{|w_n(o)|_1} \right)\\
			&= \frac{1}{|w_n(\widehat{\alpha})|_1|w_n(o)|_1} B(w_n(\widehat{\alpha}),w_n(o))\\
			&= \frac{1}{|w_n(\widehat{\alpha})|_1|w_n(o)|_1} B(\alpha,o) \rightarrow 0\ \ \ \ 
			(n\rightarrow \infty).
	\end{align*} 
	We have the conclusion.
\end{proof}


\section{The Gromov boundary and the CAT(0) boundary}

\subsection{The Gromov boundaries}
The Gromov boundary of a hyperbolic space is one of the most studied boundary
at infinity.
In this section we define it for an arbitrary metric space due to \cite{BK}.

Let $(X,d,o)$ be a metric space with a base point $o$. 
We denote simply $(* \vert *)$ as the Gromov product with respect to 
the base point $o$.
A sequence $x = \{x_i\}_i$ in $X$ is a {\it Gromov sequece} if
$(x_i\vert x_j)_z \ar \infty$ as $i,j \ar \infty$ for any base point $z \in X$.
Note that if $(x_i\vert x_j)_z \ar \infty\ (i,j \ar \infty)$ for some $z \in X$ then 
for any $z' \in X$ we have $(x_i\vert x_j)_{z'} \ar \infty\ (i,j \ar \infty)$.

We define a binary relation $\sim_G$ on the set of Gromov sequences as follows.
For two Gromov sequences $x = \{x_i\}_i, y=\{y_i\}_i$, $x \sim_G y$ if 
$\liminf_{i,j \ar \infty}(x_i \vert y_j) = \infty$.
Then we say that two Gromov sequences $x$ and $y$ are equivalent $x \sim y$
if there exist a finite sequence $\{x=x_0,\ldots,x_k=y\}$ such that 
\[
	x_{i-1} \sim_G x_i \ \ {\text{for}}\ \ i = 1,\ldots,k.
\]
It is easy to see that the relation $\sim$ is an equivalence relation on the set of 
Gromov sequences.
The {\it Gromov boundary} $\partial_G X$ is the set of 
all equivalence classes $[x]$ of Gromov sequences $x$.
If the space $X$ is a finitely generated group $G$ then the Gromov boundary of $G$ depends on the
choice of the generating set in general.
In this thesis we always define the Gromov boundary of a Coxeter group $W$
using the generating set of the Coxeter system $(W,S)$.
We shall use without comment the fact that every Gromov sequence is equivalent to each of
its subsequences.
To simplify the statement of the following definition, we denote a point $x \in X$ by 
the singleton equivalence class $[x]=[\{x_i\}_i]$ where $x_i=x$ for all $i$.  
We extend the Gromov product with base point $o$ to 
$(X\cup \partial_GX) \times ( X\cup \partial_GX)$ via the equations
\begin{align*}
	(a\vert b) &= 
	\begin{cases}
	\ \inf \left\{\left.\liminf_{i,j \ar \infty} (x_i \vert y_j )\ \right\vert\ 
						[x] = a,\ [y] = b\right\},\quad \text{if}\ a \not= b, \\
	\ \infty, \quad \text{if}\ a = b.
	\end{cases}
\end{align*}

We set
\[
	U(x,r) := \{y\in \partial_G X \ \vert\ 
	(x\vert y) > r\}
\]
for $x \in \partial_G X$ and $r>0$
and define $\mathcal{U} = \{U(x,r)\ \vert\ x \in \partial_G X, r>0\}$.
The Gromov boundary $\partial_G X$ can be regarded as a topological space 
with a subbasis $\mathcal{U}$.

If the space $X$ is $\delta$-hyperbolic in the sense of Gromov,
then this topology is equivalent to a topology defined by the following metric.
For $\epsilon > 0$ satisfying $\epsilon \delta \le 1/5$, we define $d_\epsilon$ as follows:
\[
	d_\epsilon(a,b) = e^{-\epsilon(a\vert b)} \quad (a,b \in \partial_G X).
\]
Then it follows from 5.13 and 5.16 in \cite{vaisala} that 
$d_\epsilon$ is actually a metric.
In this thesis, we always take $\epsilon$ so that $\epsilon \delta \le 1/5$
for all $\delta$ hyperbolic spaces $X$ and 
assume that $\partial_G X$ is equipped with $d_\epsilon$-topology.

If an isometric group action $G \curvearrowright X$ on a Gromov hyperbolic space $X$
is proper and cocompact then 
the group $G$ is also hyperbolic in the sense of Gromov and it is called 
a hyperbolic group (see \cite{vaisala}).


\subsection{The CAT(0) boundaries}
The map we want is given via the {\it CAT(0) boundary} $\partial_I D$ (or $\partial_ID'$)
of $D$ (or $D'$).
That is a space of geodesic rays emanating from a base point.

Assume that $(X,d)$ is a complete geodesic space. 
Fix a point $o$ in X.
We denote $GR(X,o)$ to be the set of geodesic rays emanating from $o$:
\[
	GR(X,o) := \{\gamma \in C([0,\infty),X)\ \vert\ 
			\gamma(0) = o,d(o,\gamma(t)) = len(\gamma\vert_{[0,t]})
			\ \forall t \in [0,\infty)\},
\]
where $C([0,\infty),X)$ denotes the class of continuous maps from $[0,\infty)$ to $X$.
Then we set $GR(X) := \bigcup_{o \in X} GR(X,o)$.
Two rays $\gamma,\eta \in GR(X)$ are equivalent $\gamma \sim \eta$ if the supremum  
$\sup_{t \ge 0}d(\gamma(t),\eta(t))$ is finite.
Let $\partial_{I}X$ be the coset $GR(X)/\sim$ and 
call this the ideal boundary of $X$.
If $X$ is $CAT(0)$ in addition, then for any point $\xi$ in $\partial_{I} X$ 
there exists a unique geodesic $\gamma$ emanating from $o$ so that 
the equivalence class of $\gamma$ equals to $\xi$ (consult with \cite{BH}).
Hence we can identify $GR(X,o)$ and $\partial_I X$ for some fixed $o \in X$ 
whenever $X$ is CAT(0).
In this case we call $\partial_{I} X$ {\it the CAT(0) boundary} of $X$.
Since all geodesic rays in $GR(X,o)$ are unbounded,
$\partial_{I} X$ appears at infinitely far from any point in $X$.

We assume that $(X,d)$ is complete CAT(0) space.
We attach the {\it cone topology} $\tau_C$ to the union $X \cup \partial_{I} X$ then 
it coincides with original topology in $X$.
This topology is Hausdorff and compact whenever $X$ is proper.
We omit the definition of $\tau_C$.
For the detail of the cone topology, see \cite{BH}.
This is defined by using a base point $o \in X$ but is independent of the choice of $o$.
First, notice that for any $x \in X \cup \partial_{I} X$ there exists
a unique geodesic $\gamma_x$ from $o$ to $x$.
In the case where $x \in \partial_{I} X$, we merely mean that $x$ equals to the equivalence 
class of $\gamma_x$.
For $r \in (0,\infty)$ set $X_r = \partial_{I} X \cup (X \setminus \overline{Ball(o,r)})$
where $\overline{Ball(o,r)}$ is the closure of an open ball $Ball(o,r)$
 centered at $o$ whose radius is $r$.
Let $S(o,r)$ be the boundary of $Ball(o,r)$ and let 
$p_r: X_r \ar S(o,r)$ be the projection defined by $p_r(x) = \gamma_x(r)$
and let the set $U(a,r,s),\ r,s>0$, consist of all $x \in X_r$ such that
$d(p_r(x),p_r(a)) < s$.
We notice that $U(x,r,s)$ consists of geodesics passing through the intersection of 
$S(o,r)$ and $Ball(p_r(x),s)$.
Then $\tau_C$ has as a local base at $a \in \partial_I X$ the sets $U(a,r,s),\ r,s>0$.
 
We return to our situation.
Since the region $D'$ and $D$ are both complete CAT(0) space,
CAT(0) boundaries for each space are well defined.
We use the eigenvector $o$ for the negative eigenvalue as the base point in the definition
of CAT(0) boundary and the cone topology.
Furthermore since $D'$ is a subspace of $D$, its CAT(0) boundary $\partial_I(D')$ 
is a subspace of $\partial_I D$.

\begin{prop}\label{cat}
	$\partial_{I} D$ (resp. $\partial_{I} D'$)
	is homeomorphic to 
	$\partial D$ (resp. $\partial D' \setminus D$).
\end{prop}

\begin{proof}
	It suffices to see this for the case where the entire space $D$.
	Fix a base point $o \in D$.
	For any $\xi \in \partial_I D$, $\xi$ is a geodesic ray from $o \in D$ and is 
	also a geodesic segment with respect to the Euclidean metric in $D$.
	Hence $\xi$ defines a unique endpoint $x$ in $\partial D$.
	Conversely for any $y \in \partial D$ take 
	a segment $[o,y]$ from $o$ to $y$.
	Then $[o,y]$ is a geodesic with respect to the Hilbert metric which tends to infinity.
	Therefore we have a bijection $h : \partial_I D \ar \partial D$.
	
	For any $\gamma \in \partial_I D = \partial D$
	we identify the geodesic $\gamma$ emanating from $o$ in the topology of $d$  
	and a (half-open) segment $[o,\gamma] \setminus \{\gamma\}$ in the Euclidean topology 
	parametrized by $[0,\infty)$ so that $h(\gamma(t)) = \gamma(t)$.
 
	Let $U$ be an open ball with respect to the Euclidean subspace topology
	centered at some point in $\partial D$.
	We set $\widetilde{U}=\bigcup_{\gamma \in U,\ t \in (0,\infty)} \gamma(t)$. 
	Obviously $\widetilde{U}$ is open in the Euclidean topology.
	Then for any $\gamma$ in $U$ and any $t \in [0,\infty)$ there exists
	a Euclidean open ball $Ball_E(\gamma(t))$ centered at $\gamma(t)$ 
	included in $\widetilde{U}$.
	Since the identity map $(D,d) \ar (D,d_E)$ is a homeomorphism,
	we have an open ball $Ball(\gamma(t),s)$ centered at $\gamma(t)$ 
	in $Ball_E(\gamma(t))$ with respect to the topology of $d$.
	Considering the intersection $T$ of sphere $S(o,t)$ and $Ball(\gamma(t),s)$,
	we see that geodesics from $o$ through $T$ is included in $U$.
	This shows that $h$ is a continuous bijection from a compact set to a Hausdorff space 
	and hence it is a homeomorphism.
\end{proof}

\begin{rem}\label{issho}
	If the case space $X$ is a complete proper hyperbolic CAT(0) space 
	then $\partial_G X \simeq \partial_I X$ (\cite[Theorem 2.2 (d)]{BK}).
	Because of this, if the case (i) (resp. the case (iii)) happens
	then $\partial_ID \simeq \partial_G D$ 
	(resp. $\partial_I D' \simeq \partial_G D'$).
\end{rem}

\begin{rem}\label{issho2}
	If the case (iii) happens, then $\Lambda(W)$ is homeomorphic to $\partial D' \setminus D$
	by 
	Theorem \ref{main2}.
	Together with this and Proposition \ref{cat}, we see that
	$\Lambda(W) = \partial D' \setminus D \simeq \partial_I D' \simeq \partial_G D'$.
\end{rem}


\section{The Cannon-Thurston maps}
In this section, we give a proof of Theorem \ref{main}.
Throughout this section, a vector $o$ denotes the normalized (with respect to $|*|_1$) eigenvector corresponding to 
the negative eigenvalue of $B$.

\subsection{The case of $W$ acting without cusps}
We consider when $W$ acts cocompactly or convex cocompactly.
In this case $W$ is hyperbolic in the sense of Gromov.
Moreover for the case $(iii)$, $K'$ is bounded.
Together with the convexity of $D'$,
we see that Proposition \ref{shortest} also holds in this case. 

For simplicity, we mean $\widetilde{D}$ for $D$ or $D'$.
Our purpose in this section is actually to construct a homeomorphism from
$\partial_G(W,S)$ to $\partial \widetilde{D}$ via Remark \ref{issho}, \ref{issho2}.

We define the map $f : W \ar \widetilde{D}$ by $w \mapsto w \cdot o$ 
where $o$ is the eigenvector of the negative eigenvalue.
This map is a quasi-isometry by Lemma \ref{qi}.

It is well known that $f$ extends to a homeomorphism
between $\partial_G(W,S) \cup W$ and $\partial_G \widetilde{D} \cup \widetilde{D}$
(conf. \cite{vaisala}).
Let $\overline{f}$ be the restriction of the homeomorphism above to $\partial_GW$.
Now we recall following two maps.
By the result of Buckley and Kokkendorff \cite{BK},
we know that there exists a homeomorphism 
$g : \partial_G \widetilde{D} \ar \partial_I \widetilde{D}$.
Moreover, for a Gromov sequence $\xi \in \partial_G \widetilde{D}$ 
any unbounded sequence given as a subset of a geodesic ray $g(\xi)$
is equivalent to $\xi$.
On the other hand by Proposition \ref{cat} 
we have a homeomorphism $h : \partial_I \widetilde{D} \ar \partial \widetilde{D}$.

We compose these homeomorphisms.
Let $F = h\circ g\circ \overline{f}$.
Then we have a homeomorphism from $\partial_G(W,S)$ to $\partial \widetilde{D}$.
We verify that $F$ sends $\omega \in \partial_G(W,S)$ to 
the limit point defined by $\{w_k \cdot o\}_k$ for $\{w_k\}_k \in \omega$.
If this is true, then we see that $F$ is $W$-equivariant by the construction.
To see this, we inspect the details of the maps $g$ and $h$.
For our situation, the proof in \cite{BK} says that 
for a Gromov sequence $\{w_k \cdot o\}_k \in F([\{w_k\}_k])$ in $W$, there exists a $\xi$
such that a sequence $\{u_i\cdot o\}_i$ constructed by the same way as 
in the proof of Proposition \ref{shortest} is a short sequence 
included in a bounded neighborhood of $\xi$.
The image of $\xi$ by $h$ is equivalent to $\{u_i\cdot o\}_i$ in the sense of Gromov.
Adding to this, Buckley and Kokkendorff showed that 
$\{u_i\cdot o\}_i$ equivalent to the original sequence $\{w_k\cdot o\}_k$ and hence 
they converge to the same point in $\partial_G \widetilde{D} \setminus D$.
By Remark \ref{issho2} $F$ is the map we want.


\subsection{The case of $W$ acting with cusps}
We know that there exist some Coxeter groups acting on $D$ with cusps.
By Proposition \ref{bunrui}, 
this happens when $\partial D$ is tangent to some faces of $\conv(\Delta)$.
We divide this case into following three cases;
\begin{itemize}
	\item[(i)] 
	there exists at least one pair of simple roots $\alpha, \beta \in \Delta$ 
	so that $B(\alpha,\beta) = -1$,
	\item[(ii)]
	there exists at least one subset $\Delta' \subset \Delta$ 
	whose cardinality is more than $3$
	so that the corresponding matrix $B'$ is positive semidefinite (not positive definite) 
	where $B'$ is the matrix obtained by restricting $B$ to $\Delta'$,
	\item[(iii)]
	or (i) and (ii) happen simultaneously.
\end{itemize}
We only deal with the case (i).
In this case, the dihedral subgroup of $W$ generated by $s_\alpha$ and $s_\beta$ 
is infinite and its limit set is one point.
This means that $D$ is tangent to the segment connecting $\alpha$ and $\beta$.
Hence the fundamental region of $W$ is unbounded.

For the cases (ii) and (iii), 
we have to see other geometric aspects of the Coxeter groups.
We will discuss the existence of the Cannon-Thurston maps 
for the excepted cases in the next section.

Recall that the number $n$ is the rank of $W$ and hence equals to the dimension of $V$. 
Let $\{A_m\}_m$ be a sequence of $n \times n$ matrices 
which are defined as follows.
For each $m \in \N$, we define $A_m$ so that
\[
	A_m(\alpha,\beta) = 
		\begin{cases}
			\  1/m, \quad & \text{if}\ B(\alpha,\beta) = -1,\\
			\  0,         & \text{if}\ \text{otherwise},
		\end{cases}
\]
for each $\alpha,\beta \in \Delta$.
We denote the bilinear form with respect to each $A_m$ by $A_m(v,v')$ for $v,v' \in V$.
Then let $B_m = B - A_m$.

If $B$ has the signature $(n-1,1)$,
then $B_m$ also has the signature $(n-1,1)$ for sufficiently large $m \in \N$.
Therefore for sufficiently large $m$,
our definitions of $Q$,$D$, $D'$, $L$, $K$
can be extended to the bilinear form defined by $B_m$.
We define $Q_m$, $D_m$, $D'_m$, $L_m$, $K_m$ each of them by 
using $B_m$ instead of $B$ in their definitions.
Clearly $B_m$ converges to $B$ as $m$ tends to $\infty$.

Let $v_1,\ldots,v_n$ be eigenvectors of $B$ normalized with respect to the Euclidean norm
so that the matrix $(v_1,\ldots,v_n)$ diagonalize $B$.
Then since each $P_{m,i}(v_i)$ converges to $v_i$,
the matrix diagonalizing $B_m$ also converges to $(v_1,\ldots,v_n)$.
This fact shows that the sequence $\{D_m\}_m$ converges to $D$.

We can consider the $B_m$-reflection of $W$ on $V$ with respect to $B_m$.
We denote this action by $\rho_m$.
For example, the $B_m$-reflection of $\alpha \in \Delta$ can be calculated as
\[
	\rho_m(s_\alpha) (x) = x - 2B_m(x,\alpha)\alpha,\quad (x \in V). 
\]
The normalized action with respect to $B_m$ is defined in the same way as $B$.
We denote this also by $\rho_m$.
Furthermore if $B_m$ has the signature $(n-1,1)$,
then all our lemmas and propositions can be proved
by using the normalized eigenvector $o_m$
corresponding to the negative eigenvalue of $B_m$ instead of $o$.
Therefore if the normalized action $\rho_m$ is (convex) cocompact, 
then there exists a map $F_m$ from the Gromov boundary $\partial_G(W,S)$ of $W$ to
the limit set $\Lambda_{B_m}(W)$ which is homeomorphic. 
In fact we have a $W$-equivariant homeomorphism $F_m : \partial_G(W,S) \ar \Lambda_{B_m}(W)$
for each $m$ since the case (iii) happens.
Note that for sufficiently large $m$, we have $V_0 \cap Q_m = \{{\bf 0}\}$.
Hence we can define the Hilbert metric on $V_1 \cap {Q_m}_-$ where 
${Q_m}_- = \{v \in V\ \vert\ B_m(v,v)<0\}$.
Consider the correspondence between $x \in D_m$ and $y = \R x \cap V_1 \cap {Q_m}_-$.
Then we see that this is an isometry between $D_m$ and $V_1 \cap {Q_m}_-$ and
$W$ equivaliant.
Thus we can regard the normalized action $\rho_m$ as an action of $W$ on $V_1 \cap {Q_m}_-$.

We remark that for any $\alpha \in \Delta$ and $m \in \N$, we have 
$B_m(o,\alpha) = B(o,\alpha) - A_m(o,\alpha) < 0$ since $B(o,\alpha) < 0$ and all coordinates
of $o$ are positive.
Hence $o$ is in $K_m$ for any $m \in \N$.

\begin{lem}\label{kuwashiku}
	Let $o$ be the normalized eigenvector of the negative eigenvalue of $B$.
	There exists a constant $C_1>0$ such that $|w(o)|_1 \ge C_1 |w|$ for any $w \in W$.
\end{lem}

\begin{proof}
	Let $\lambda > 0$ be the absolute value of the negative eigenvector of $B$ 
	hence $B o = -\lambda o$. 
	Note that all coordinates of $o$ are positive by Lemma \ref{korekore}.
	If $|w|=1$ then there exists $\alpha \in \Delta$ such that $w = s_\alpha$.
	Then we have 
	\[
		|s_\alpha (o)|_1 = |o|_1 -2B(o,\alpha)|\alpha|_1
						 = 1 + 2\lambda ( o,\alpha )|\alpha|_1 > 1 = |s_\alpha|.
	\]

	Before moving to the inductive step we remark the following.
	By \cite[Lemma 2.10 (ii)]{hlr} there exists a constant $C'$ such that 
	for $w \in W$ and $\alpha \in \Delta$ with $w(\alpha) \in \mathrm{cone}(\Delta)$,
	 $|w(\alpha)|_1 \ge C'|w|^{\frac{1}{2}}$
	where $\mathrm{cone}(\Delta)$ is the cone spanned by $\Delta$.
	Since $o (= \sum_{\delta \in \Delta} o_\delta \delta)$ is in the convex hull of $\Delta$  
	each coordinate $o_\delta$ of $o$ satisfies $0\le o_\delta \le 1$.
	Letting $\lambda' = \min_{\delta \in \Delta} o_\delta$, 
	for $w \in W$ and $\alpha \in \Delta$ with $w(\alpha) \in \mathrm{cone}(\Delta)$ we have  
	\[
		-B(o,w(\alpha))= \lambda\sum_\beta o_\beta {w(\alpha)}_\beta 
					   \ge \lambda\lambda' |w(\alpha)|_1
					   \ge \lambda\lambda' C'|w|^{\frac{1}{2}}
	\]
	where ${w(\alpha)}_\beta$ denotes the $\beta$-th coordinate of $w(\alpha)$.
		
	For the inductive step we take an arbitrary $w \in W$ with $|w| = k+1$ ($k \in \N$) and 
	assume that for any $w' \in W$ with $|w'|\le k$ we have $|w'(o)|_1 \ge C|w'|$ for some 
	universal constant $C\ge1$.
	We take $w'\in W$ so that $w = s_\alpha w'$ with $|w'| = k$ for some $\alpha \in \Delta$.
	Then by Remark \ref{bb} all coordinates of ${w'}^{-1}(\alpha)$ are non-negative.
	From the argument above, we have
	\[
		|w(o)|_1 =   |w'(o)|_1 -2B(o,{w'}^{-1}(\alpha))
				 \ge Ck + 2 \lambda\lambda' C'k^{\frac{1}{2}} 
				 \ge C(k+1).
	\]
	if $C \le 2 \lambda\lambda'C'$.
	By taking $C_1$ so that $C_1 \le \min\{1, C\}$,
	the conclusion follows.
\end{proof}

Let $c_0 > 1$ be the maximum operator norm of $S$. 
More precisely, we set
$
	c_0=\max_{s \in S} \max_{x \in S^{n-1}} \|s(x)\|
$
where $S^{n-1}$ is the sphere in $V$ centered at $0$.
Then for any $w \in W$ with $|w| =k$, 
we have $c^k \ge \|w(o)\|$.
Since the Euclidean norm $\|*\|$ is comparable to $|*|_1$ in the cone $Q_-^+$,
there exists a constant $C_{2,0}$ such that $C_{2,0}c_0^k \ge |w(o)|_1$.
We can take these constants $C_{2,m}$ and $c_m$ for each $\rho_m(W)$ ($m\in \N$).
Since the sequence $\{B_m\}_m$ converges to $B$,
sequences $\{C_{2,m}\}_m$ and $(c_m)$ must converge to $C_{2,0}$ and $c_0$.
Thus there must exist the maximum
\[
	C_2 = \max_{m \in \N\cup\{0\}}C_{2,m},\quad \text{and}\quad c = \max_{m \in \N\cup\{0\}}c_m.
\]

\begin{prop}
	Assume that the normalized action of $W$ includes rank $2$ cusps.
	There exists a continuous $W$-equivariant surjection $\iota:\Lambda(\rho_1(W)) \ar \Lambda(W)$.
\end{prop}

\begin{proof}
	Since $Q$ and $V_0$ meet only at $0$, $B$ is positive definite on $V_0$.
	Hence $B$ defines an inner product on $V_0$ and it gives a metric on $V_0$
	by $q(x-y)^{\frac{1}{2}}$.
	It is easy to see that this metric induces to $V_1$ and it is comparable to 
	the Euclidean metric.
	
	Let $o$ be the normalized eigenvector for the negative eigenvalue $-\lambda$ of $B$.
	Notice that $o \in K_m$ for any $m \in \N$ since $B_m(o,\alpha) = B(o,o) -A_m(o,\alpha)<0$.
	We claim that for any short sequence $\{w_k\}_k$ in $W$,
	if $\rho_m(w_k)\cdot o \ar \xi \in \partial D$ as $k,m \ar \infty$ then
	$w_k\cdot o \ar \xi$ as $k \ar \infty$.
	This ensures that the correspondence $\iota(\xi_1) = \xi$ for each 
	$\xi_1 \in \Lambda(\rho_1(W))$ is actually a map where $\xi \in \Lambda(W)$ 
	is the equivalence class of the sequence $\{w_k\cdot o\}_k$ for $\{w_k\}_k$ defining $\xi_1$.
	If $\iota$ is well-defined then it is obviously $W$-equivariant and surjective.
	To show the continuity of $\iota$, it suffices to see that
	$ q(w\cdot o-\rho_m(w)\cdot o) \ar 0$ as $k,m \ar \infty$ uniformly.
		
	Fix $m \in \N$ arbitrarily.
	For any $x \in \mathrm{cone}(\Delta)$ and any $\alpha \in \Delta$ we have
	\[
		|\rho_m(s_\alpha)(x)|_1 = |s_\alpha(x) + 2A_m(x,\alpha)\alpha)|_1 \ge |s_\alpha(x)|_1
		\qquad \text{and} \qquad
		A_m(x,\alpha) \le \frac{|x|_1}{m}.
	\]
	The first inequality shows that for any $x \in D$ 
	whose orbit $W(x)$ is included in $\mathrm{cone}(\Delta)$,
	we have $|\rho_m(w)(x)|_1 \ge |w(x)|_1$ for any $w \in W$.
	Since $B(s_\alpha(x),\alpha) = -B(x,\alpha)$, the second inequality implies the following;
	\begin{align*} 
		-B(s_\alpha(x),\rho_m(s_\alpha)(x)) 
		&=   -B(s_\alpha(x), s_\alpha(x) + 2A_m(x,\alpha)\alpha)\\
		&\le -q(x) + 2A_m(x,\alpha)B(x,\alpha)\\
		&\le -q(x) + 2C\frac{|x|_1}{m}B(x,\alpha),
	\end{align*}
	for any $x \in \mathrm{cone}(\Delta)$ and any $\alpha \in \Delta$
	where $C$ is a constant depending on $B$.

	Now we claim that $-B(w(o), \rho_m(w)(o)) \le -q(o)$ for any $w \in W$. 
	We show this by the induction for the word length.  
	If $|w| = 1$ then $w = s_\alpha$ for some $\alpha \in \Delta$.
	Hence the argument above gives
	\[
		-B(s_\alpha(o),\rho_m(s_\alpha)(o)) \le -q(o) - \frac{2C\lambda o_\alpha}{m} \le -q(o),
	\]
	where $o_\alpha$ denotes the $\alpha$-th coordinate of $o$.	
	For any $w \in W$ satisfying the condition $-B(w(o),\rho_m(w(o)) \le -q(o)$,
	if $|s_\alpha w| = |w| + 1$ for $\alpha \in \Delta$ then 
	\begin{align*}
		&-B(s_\alpha w(o),\rho_m(s_\alpha)\rho_m(w)(o))\\
			&\quad\quad\quad\quad
			= -B(s_\alpha w(o),s_\alpha(\rho_m(w)(o))+2A_m(\rho_m(w)(o),\alpha)\alpha)\\
			&\quad\quad\quad\quad
			= -B(s_\alpha w(o),s_\alpha(\rho_m(w)(o)))
				-2A_m(\rho_m(w)(o),\alpha)B(s_\alpha w(o),\alpha)\\
			&\quad\quad\quad\quad
			= -B(s_\alpha w(o),s_\alpha(\rho_m(w)(o)))
				+2A_m(\rho_m(w)(o),\alpha)B(w(o),\alpha)\\
			&\quad\quad\quad\quad
			= -B(s_\alpha w(o),s_\alpha(\rho_m(w)(o)))
				-2\lambda A_m(\rho_m(w)(o),\alpha)(o,w^{-1}(\alpha))\\
			&\quad\quad\quad\quad
			\le-q(o),
	\end{align*}
	where $(,)$ denotes the Euclidean inner product.
	Furthermore we have
	\begin{align*}
		q(\rho_m(w)(o))
			&= B(\rho_m(w)(o),\rho_m(w)(o))\\
			&= B_m(\rho_m(w)(o),\rho_m(w)(o))+A_m(\rho_m(w)(o),\rho_m(w)(o))\\
			&= B_m(o,o)+A_m(\rho_m(w)(o),\rho_m(w)(o))\\
			&= q(o)-A_m(o,o)+A_m(\rho_m(w)(o),\rho_m(w)(o))\\
			&\le q(o)+(Cq(o))^2 \frac{|\rho_m(w)(o)|_1^2}{m}
	\end{align*}
	for any $w \in W$.
	Putting these inequalities together we deduce that for $w \in W$ with $|w| = k$,
	\begin{align*}
		q(w\cdot o-\rho_m(w)\cdot o) 
			&= \frac{q(o)}{|w(o)|_1^2}-2\frac{B(w(o), \rho_m(w)(o))}{|w(o)|_1|\rho_m(w)(o)|_1}
				+\frac{q(\rho_m(w)(o))}{|\rho_m(w)(o)|_1^2}\\
			&\le \frac{q(o)}{|w(o)|_1^2}-2\frac{q(o)}{|w(o)|_1|\rho_m(w)(o)|_1}
				+\frac{q(o)}{|\rho_m(w)(o)|_1^2}+\frac{1}{m}\\
			&\le \frac{q(o)}{C_1k^2}-2\frac{q(o)}{C_2k^{2}}
				+\frac{q(o)}{C_1k^2}+\frac{1}{m}.
	\end{align*}
	This shows that the convergence of $ q(w\cdot o-\rho_m(w)\cdot o) \ar 0$ as $k,m \ar \infty$
	does not depend on the short sequence $(w_k)$.
	Thus $\iota$ is well-defined and continuous.
\end{proof}

Considering the composition $F' = \iota \circ F_1$,
we have the map which is surjective, continuous and $W$-equivariant.

If $B(\alpha,\beta) = -1$ for some $\alpha,\beta \in \Delta$ 
then the Coxeter subgroup $W'$ generated by $\{s_\alpha,s_\beta\}$ is affine.
Since an affine Coxeter group has only one limit point, 
$\{(s_\alpha s_\beta)^k\cdot o\}_k$ and $\{(s_\beta s_\alpha)^k \cdot o\}_k$ converges to the same 
limit point.
However in the Gromov boundary of $(W,S)$,
$\{(s_\alpha s_\beta)^k\}_k$ and $\{(s_\beta s_\alpha)^k\}_k$ lie in distinct equivalence classes.
In fact, considering another action of $(W,S)$ defined by 
another bi-linear form $B'$ such that $B'(\alpha,\beta) < -1$,
then the limit set $\Lambda_{B'}(W') \subset \Lambda_{B'}(W)$ consists of two points.
In this case the limit points of
$\{(s_\alpha s_\beta)^k\cdot o\}_k$ and $\{(s_\beta s_\alpha)^k \cdot o\}_k$
are distinct.
On the other hand the map $\partial_G(W,S) \ar \Lambda_{B'}(W)$ is well defined 
hence $F'$ cannot be an injection.

\end{document}